\documentclass[12pt, oneside, psamsfonts]{amsart}

\newif\ifPDF
\ifx\pdfoutput\undefined\PDFfalse
\else \ifnum \pdfoutput > 0 \PDFtrue
        \else \PDFfalse
        \fi
\fi

\usepackage[centertags]{amsmath}
\usepackage{amsfonts}
\usepackage{mathrsfs}
\usepackage{textcomp}
\usepackage{amssymb}
\usepackage{amsthm}
\usepackage{newlfont}
\usepackage[all]{xy}

\ifPDF
  \usepackage[pdftex]{color, graphicx}
  \usepackage[pdftex, bookmarks, colorlinks]{hyperref}
  \hypersetup{colorlinks=false}

\else
  \usepackage{color}
  \usepackage[dvips]{graphicx}
  \usepackage[dvips]{hyperref}
\fi

\usepackage[scale=0.8]{geometry}

\usepackage[pagewise, mathlines, displaymath]{lineno}

\newtheorem{thm}{Theorem}[section]
\newtheorem{cor}[thm]{Corollary}
\newtheorem{lem}[thm]{Lemma}
\newtheorem{prop}[thm]{Proposition}
\theoremstyle{definition}
\newtheorem{defn}[thm]{Definition}
\theoremstyle{remark}
\newtheorem{rem}[thm]{Remark}
\newtheorem{example}[thm]{Example}
\numberwithin{equation}{section}

\newcommand{\norm}[1]{\| #1\|}
\newcommand{\abs}[1]{| #1 |}

\newcommand{\Real}{\mathbb R}
\newcommand{\Int}{\mathbb Z}
\newcommand{\Comp}{\mathbb C}

\newcommand{\eps}{\varepsilon}

\newcommand{\Kzero}{\textrm{K}_0}
\newcommand{\Kone}{\textrm{K}_1}
\newcommand{\tr}{\mathrm{T}}

\begin{document}


\title{On the small boundary property and $\mathcal Z$-absorption}

\author{George A. Elliott}
\address{Department of Mathematics, University of Toronto, Toronto, Ontario, Canada~\ M5S 2E4}
\email{elliott@math.toronto.edu}

\author{Zhuang Niu}
\address{Department of Mathematics and Statistics, University of Wyoming, Laramie, Wyoming 82071, USA}
\email{zniu@uwyo.edu}

\keywords{$\mathcal Z$-Absorbing C*-Algebras, Small Boundary Property, Mean Topological Dimension}
\date{\today}
\subjclass{46L35, 37B02}


\begin{abstract}
We introduce Property (C) for a unital commutative sub-C*-algebra $D$ of a unital C*-algebra $A$, which is a version of the relative comparison property using almost normalizers. In the case that $D = \mathrm{C}(X)$ and $A = \mathrm{C}(X) \rtimes\Gamma$, where $(X, \Gamma)$ is a free and minimal dynamical system with uniform Rokhlin property, it turns out that this Property (C) is equivalent to the small boundary property of $(X, \Gamma)$.


\end{abstract}

\maketitle

\setcounter{tocdepth}{1}


\section{Introduction}

This is a continuation of our study \cite{EN-SBP-I} of the relation between the small boundary property of dynamical systems and the $\mathcal Z$-absorption of C*-algebras.

The Jiang-Su algebra $\mathcal Z$ is an infinite-dimensional unital simple separable amenable C*-algebra which has the same value of the Elliott invariant as $\Comp$. A C*-algebra $A$ is said to be $\mathcal Z$-absorbing if $A \cong A \otimes \mathcal Z$, and the class of $\mathcal Z$-absorbing C*-algebras (which includes $\mathcal Z$ itself) is considered to be well behaved. In fact, the class of $\mathcal Z$-absorbing unital simple separable amenable C*-algebras which satisfy the Universal Coefficient Theorem of KK-theory (possibly redundant) is  classified by the conventional Elliott invariant (see \cite{GLN-TAS-1}, \cite{GLN-TAS-2}, \cite{EN-K0-Z}, \cite{EGLN-DR}, \cite{EGLN-ASH}, \cite{TWW-QD}, \cite{CETWW-dim-n}).

On the other hand, the small boundary property of a topological dynamical system was introduced in \cite{Lindenstrauss-Weiss-MD} as a dynamical system analogue of the usual definition of zero-dimensional space. It is closely related to the mean (topological) dimension, which was introduced by Gromov (\cite{Gromov-MD}), and then was developed and studied systematically by Lindenstrauss and Weiss (\cite{Lindenstrauss-Weiss-MD}), as a numerical invariant which measures the complexity of a dynamical system in terms of the dimension growth with respect to partial orbits. The small boundary property always implies the zero mean dimension (\cite{Lindenstrauss-Weiss-MD}), and the converse holds for $\mathbb Z^d$-actions with the marker property (\cite{Lind-MD}, \cite{GLT-Zk}).

It was shown in \cite{EN-MD0} that the small boundary property of $(X, \Int)$ implies $\mathcal Z$-absorption of the crossed product C*-algebra $ A = \mathrm{C}(X) \rtimes \Int$. In this paper, we investigate the converse: how to recover the small boundary property of $(X, \Int)$ in terms of C*-algebra information.

We consider the following property for a pair of C*-algebras $D \subseteq A$:  
\newtheorem*{defn-n}{Definition}
\begin{defn-n}[Definition \ref{defn-C}]
Let $A$ be a unital C*-algebra and let $D$ be a unital commutative sub-C*-algebra. Then the pair $(D, A)$ is said to have Property (C) if for any positive contractions $f, g, h \in D$ satisfying $f, g \in \overline{hDh}$, and $$\mathrm{d}_\tau(f) < \mathrm{d}_\tau(g),\quad \tau \in \tr(A),$$ and for any $\eps>0$, there is a contraction $u \in \overline{hAh} + \Comp 1_A$ such that
$$ u f u^* \in_\eps^{\norm{\cdot}_2} \overline{gAg},$$
$$\mathrm{dist}_{2, \tr(A)}(udu^*, (D)_1) < \eps,\quad \mathrm{dist}_{2, \tr(A)}(u^*du, (D)_1) < \eps, \quad d \in (D)_1, $$ 
and
$$\norm{uu^* - 1}_{2, \tr(A)},\  \norm{u^*u - 1}_{2, \tr(A)} < \eps.$$
\end{defn-n}

Property (C) can be regarded as a relative comparison property for $D$ inside $A$, with respect to the uniform trace norm, but requires the comparison being implemented by almost normalizers. This property for C*-pair turns out to be equivalent to the small boundary property of the dynamical system $(X, \Gamma)$: 


\begin{thm}[Corollary \ref{Z-SBP}]
Let $(X, \Int^d)$ be a free and minimal topological dynamical system, and let $A = \mathrm{C}(X)\rtimes \Int^d$. Then $(\mathrm{C}(X), A)$ has Property (C) if, and only if $(X, \Int^d)$ has the (SBP). 
\end{thm}

This theorem also holds for a free and minimal action of an arbitrary  amenable group with the uniform Rokhlin Property (Definition \ref{defn-URP-COS}). 

In order to show the (SBP) for the pair $(D, \tr(A))$, by \cite{EN-SBP-I}, it is enough to show that every self-adjoint element of $D$ can be approximated by self-adjoint elements of $D$ with a neighbourhood of $0$ uniformly small under all tracial spectral measures (see Theorem \ref{SBP-2-norm}). Such approximating elements exist (Property (S)), but only in the ambient C*-algebra $A$. However, this can be fixed by using Properties (C): Upon using  Property (E) (Definition \ref{defn-D}), an existence property which always holds for the C*-algebras in question, one obtains a self-adjoint element in the subalgebra $D$ which almost has the same trace spectral measure distributions as the self-adjoint element in $A$ provided by Property (S), and then upon using Property (C), this element can be twisted inside $D$ to approximate the given self-adjoint element and still have a neighbourhood of $0$ uniformly small under all tracial spectral measures. This shows the (SBP) for $(D, \tr(A))$.

Nevertheless, Property (C) still involves the pair of C*-algebras $D \subseteq A$. The question is still open whether the small boundary property of $(X, \Gamma)$ can be recovered just from the ambient C*-algebra $A$, or in particular, whether $A \cong A\otimes \mathcal Z$ implies that the small boundary property of $(X, \Gamma)$. This question seems to be close related to  possible uniqueness of the Cartan subalfgebras of $A$ with respect to the uniform trace norm.

\subsection*{Acknowledgements} The research of the first named author was supported by a Natural Sciences and Engineering Research Council of Canada (NSERC) Discovery Grant, and the research of the second named author was supported by a Simons Foundation grant  (MP-TSM-00002606). Part of the work was carried out during the visit of the second named author to Kyushu University in December 2024; he  thanks Yasuhiko Sato for the discussions and the hospitality during the visit. AI is not used during the research of this work.

\section{Preliminaries and notation}

In this section, let us collect some notation and definitions concerning C*-algebras and dynamical systems.

\subsection{Comparison of positive elements and $\mathcal Z$-absorbing C*-algebras}

Let $A$ be a C*-algebra, and let $a, b \in A$ be positive elements. One says that $a$ is Cuntz subequivalent to $b$, denoted by $a \precsim b$, if there is a sequence $(x_n)$ in $A$ such that
$$ \lim_{n \to \infty} x_n^*b x_n = a.$$
The following lemma will be frequently used.
\begin{lem}[\cite{RorUHF2}]
If $a, b\in A$ are positive elements such that $\norm{a - b} < \eps$, then $(a-\eps)_+ \precsim b$, where $(a-\eps)_+ = f(a)$ with $f(t) = \max\{t-\eps, 0\}$, $t \in \Real$.
\end{lem}

Let $\tau \in \tr(A)$. For each positive element $a \in A$, define
$$\mathrm{d}_\tau(a) = \lim_{n\to\infty} \tau(a^{\frac{1}{n}}).$$
Then, if $a \precsim b$, one has $$\mathrm{d}_\tau(a) \leq \mathrm{d}_\tau(b),\quad \tau \in \tr(A).$$ The converse in general does not hold. 

\begin{defn}[\cite{JS-Z}]
The Jiang-Su algebra $\mathcal Z$ is the (unique) simple unital inductive limit of dimension drop C*-algebras such that $$(\Kzero(\mathcal Z), \Kzero^+(\mathcal Z), [1_\mathcal Z]_0) \cong (\Int, \Int^+, 1),\quad  \Kone(\mathcal Z) = \{0\},\quad \mathrm{and} \quad \tr(\mathcal Z) = \{\mathrm{pt}\}.$$ 

A C*-algebra $A$ is said to be $\mathcal Z$-absorbing if $A \cong A \otimes \mathcal Z$.
\end{defn}

If $A$ is simple and $\mathcal Z$-absorbing, then the strict order induced by the Cuntz subequivalence relation is determined by the rank functions; that is, for any positive elements $a, b \in A$,
$$ \mathrm{d}_\tau(a) < \mathrm{d}_\tau(b),\quad \tau \in \tr(A) \quad \Longrightarrow \quad  a \precsim b.  $$
The Toms-Winter conjecture asserts that strict comparison implies $\mathcal Z$-absorption for simple separable amenable C*-algebras (this was verified for C*-algebras with finitely many extreme traces in \cite{Matui-Sato-CP}, and then it was generalized to C*-algebras with extreme traces being compact and finite dimensional; see \cite{Sato-CP}, \cite{KR-CenSeq}, \cite{TWW-Z}).

The class of simple separable amenable $\mathcal Z$-absorbing C*-algebras which satisfy the Universal Coefficient Theorem can be classified by the conventional Elliott invariant, which in the unital case consists of the K-groups and the pairing with the trace simplex (the order on the K-group, not redundant in more general cases, is determined by the pairing) (see \cite{GLN-TAS-1}, \cite{GLN-TAS-2}, \cite{EN-K0-Z}, \cite{EGLN-DR}, \cite{EGLN-ASH}, \cite{TWW-QD}, \cite{CETWW-dim-n}):
\begin{thm}\label{classification-Z}
Let $A, B$ be unital simple separable amenable $\mathcal Z$-absorbing C*-algebras which satisfy the UCT. Then $$ A \cong B \quad \Longleftrightarrow \quad \mathrm{Ell}(A) \cong \mathrm{Ell}(B), $$ where $\mathrm{Ell}(\cdot)$ denotes the Elliott invariant. Moreover, any isomorphism between the Elliott invariant can be lifted to an isomorphism between the C*-algebras.
\end{thm}

\subsection{Uniform trace norm}

\begin{defn}
Let $A$ be a unital C*-algebra, and let $\tau \in \tr(A)$. Define
$$  \norm{a}_{2, \tau} = (\tau(a^*a))^{\frac{1}{2}}, \quad a \in A.$$
For any set $\Delta \subseteq \tr(A)$,  define the uniform trace norm
$$ \norm{a}_{2, \Delta} = \sup\{ \norm{a}_{2, \tau}: \tau \in \Delta\},\quad a \in A.$$
\end{defn}
The uniform trace norm satisfies
$$ \norm{ab}_{2, \Delta} \leq \min\{ \norm{a} \norm{b}_{2, \Delta}, \norm{a}_{2, \Delta} \norm{b}\} \quad \mathrm{and}\quad \abs{\tau(a)} \leq \norm{a}_{2, \Delta},\quad  a, b \in A,\ \tau \in \Delta.$$ 

We shall use $l^\infty(A)$ to denote the C*-algebra of bounded sequences of $A$, i.e., $$l^\infty(A) = \{(a_n): a_n \in A,\ \sup\{\norm{a_n}: n=1, 2, ...\} < +\infty \}.$$
Let $\omega$ be a free ultrafilter; then the trace-kernel is the ideal
$$J_{2, \omega, \tr(A)}:=\{(a_n) \in l^\infty(A): \lim_{n \to\omega} \norm{a_n}_{2, \tr(A)} = 0\}.$$

\begin{defn}[Definition 2.1 of \cite{CETW-Gamma}]\label{UGamma-defn}
A unital C*-algebra $A$ is said to have uniform property $\Gamma$ if for each $n\in \mathbb N$, there is a partition of unity $$p_1, p_2, ..., p_n \in (l^\infty(A)/J_{2, \omega, \Delta}) \cap A'$$ such that
$$\tau(p_iap_i) = \frac{1}{n}\tau(a),\quad a\in A,\ \tau\in \tr(A)_\omega,$$
where $\tr(A)_\omega$ denotes the set of limit traces of $l^\infty(A)$, i.e., the traces of the form $$\tau((a_i)) = \lim_{i \to \omega} \tau_i(a_i),\quad \tau_i \in \tr(A),$$
and $a$ is regarded as the constant sequence $(a) \in l^\infty(A)$.

All $\mathcal Z$-absorbing C*-algebras have uniform property $\Gamma$ (see Theorem 5.6 of \cite{CETW-Gamma}). (Indeed, all unital simple  amenable C*-algebras with unique trace have (uniform) property $\Gamma$, and if a C*-algebra $U$ has uniform property $\Gamma$, then the tensor product C*-algebra $A \otimes U$ has uniform property $\Gamma$ (an extreme trace on a tensor product is a product trace).)
\end{defn}

\subsection{AH algebras with diagonal maps}

\begin{defn}
An AH algebra with diagonal maps is the limit of an inductive sequence
$$\xymatrix{
A_1 \ar[r] & A_2 \ar[r] & \cdots \ar[r] & A = \varinjlim A_n,
}$$
where $A_i = \bigoplus_{j} \mathrm{M}_{n_{i, j}}(\mathrm{C}(X_{i, j}))$, and each connecting map preserves the diagonal subalgebras, i.e., it has the form $$f \mapsto \mathrm{diag}\{f\circ\lambda_1, ..., f\circ\lambda_m\},$$ where the $\lambda$s are continuous maps between the $X$s.
\end{defn}

All simple unital AH algebras with diagonal maps have stable rank one (\cite{EHT-sr1}), but not all AH algebras with diagonal maps are $\mathcal Z$-absorbing. In the pioneering work \cite{Vill-perf},  Villadsen constructed simple AH algebras with diagonal maps which have perforation in the ordered $\Kzero$-group. The construction was then used in \cite{Toms-Ann} to obtain a simple AH algebra with diagonal maps which has  the same value of the conventional Elliott invariant as an AI algebra, but is not isomorphic to this AI algebra. Although Villadsen algebras are not $\mathcal Z$-absorbing, a preliminary classification is obtained in \cite{ELN-Vill}.

\subsection{The small boundary property and the mean dimension}

\begin{defn}
A topological dynamical system $(X, \Gamma)$ is free if $x \gamma = x$ implies $\gamma = e$ where $x\in X$ and $\gamma \in \Gamma$. It is said to be minimal if the only closed invariant subspaces of $X$ are $\O$ and $X$. 

A topological dynamical system induces an action of $\Gamma$ on $\mathrm{C}(X)$ by $$\gamma(f)(x) = f(x\gamma),\quad x \in X.$$ We shall assume $\Gamma$ is discrete. The (universal) crossed product C*-algebra $\mathrm{C}(X) \rtimes \Gamma$ is the universal C*-algebra generated by $\mathrm{C}(X)$ and unitaries $u_\gamma$, $\gamma \in \Gamma$, with respect to the relations
$$ u^*_{\gamma} f u_{\gamma} = \gamma(f) \quad \mathrm{and} \quad u_{\gamma_1} u^*_{\gamma_2} = u_{\gamma_1 \gamma^{-1}_2},\quad f\in \mathrm{C}(X),\ \gamma, \gamma_1, \gamma_2 \in \Gamma.$$
\end{defn}

If $\Gamma$ is amenable and the dynamical system $(X, \Gamma)$ is free and minimal, the C*-algebra $\mathrm{C}(X) \rtimes \Gamma$ is simple, unital, amenable, stably finite, and satisfies the UCT. However, the C*-algebra $\mathrm{C}(X) \rtimes \Gamma$ may fail to be $\mathcal Z$-absorbing, even for $\Gamma =\Int$ (\cite{GK-Dyn}).

Let us consider the following property of dynamical systems.
\begin{defn}
A topological dynamical system $(X, \Gamma)$ is said to have the small boundary property (SBP) if for any $x \in X$ and any open neighbourhood $U$ of $x$, there is a neighbourhood $V$ of $x$ such that $V \subseteq U$ and $\mu(\partial V) = 0$ for all invariant measures $\mu$ (\cite{Lindenstrauss-Weiss-MD}).

More generally, consider a metrizable compact space $X$ and a collection $\Delta$ of Borel probability measures on $X$. The pair $(X, \Delta)$ is said to have the (SBP) if for any $x \in X$ and any open neighbourhood $U$ of $x$, there is a neighbourhood $V$ of $x$ such that $V \subseteq U$ and $\mu(\partial V) = 0$ for all $\mu \in \Delta$.  (\cite{EN-SBP-I})
\end{defn}

We have the following criterion for the (SBP):
\begin{thm}[Theorem 2.9 of \cite{EN-SBP-I}]\label{SBP-2-norm}
Let $X$ be a metrizable compact space, and let $\Delta$ be a compact set of Borel probability measures on $X$. Then the pair $(X, \Delta)$ has the (SBP) if, and only if, for any continuous real-valued function $f: X \to \Real$ and any $\eps>0$, there is a continuous real-valued function $g: X\to \Real$ such that
\begin{enumerate}
\item $\norm{f - g}_{2, \Delta} < \eps$, and 
\item there is $\delta>0$ such that $\tau_\mu(\chi_{\delta}(g)) < \eps$, $\mu \in\Delta$,
where $\tau_\mu$ is the tracial state of $\mathrm{C}(X)$ induced by $\mu$, and
\begin{equation*}
\chi_\delta(t) = \left\{ \begin{array}{ll}
1, & \abs{t} < \delta, \\
2-\abs{t}/\delta, & \delta\leq \abs{t} < 2\delta, \\
0, & \mathrm{otherwise}.
\end{array}\right.
\end{equation*}
\end{enumerate} 
\end{thm}

Mean topological dimension was introduced by Gromov (\cite{Gromov-MD}), and then was developed and studied systematically by Lindenstrauss and Weiss (\cite{Lindenstrauss-Weiss-MD}):
\begin{defn}
Consider a topological dynamical system $(X, \Gamma)$, where $\Gamma$ is discrete and amenable. Its mean dimension is defined as
$$\mathrm{mdim}(X, \Gamma):= \sup_{\mathcal U} \lim_{n \to\infty} \frac{1}{\abs{\Gamma_n}}\mathcal D( \bigwedge_{\gamma \in \Gamma_n} \mathcal U \gamma^{-1} ),$$
where $\Gamma_1, \Gamma_2, ...$ is a F{\o}lner sequence of $\Gamma$, the supremum is taken over all finite open covers $\mathcal U$ of $X$, and $\mathcal D(\mathcal U) = \min\{\mathrm{ord}(\mathcal V): \mathcal V \prec \mathcal U\}$ ($\mathrm{ord}$ is the maximal number of the mutually overlapping sets minus $1$).
\end{defn}

By \cite{Lindenstrauss-Weiss-MD}, the small boundary property of $(X, \Gamma)$ implies zero mean dimension. The converse was shown in \cite{MR3614036} and \cite{GLT-Zk} for $\Gamma = \Int^d$, and in \cite{Niu-MD-Zd} for actions with the (URP). 

Zero mean dimension (or small boundary property) implies the $\mathcal Z$-absorption of the C*-algebra:
\begin{thm}[\cite{EN-MD0}]\label{old-thm}
Let $(X, \Int)$ be a free and minimal dynamical system. If $\mathrm{mdim}(X, \Int) = 0$, then the C*-algebra $\mathrm{C}(X) \rtimes \Int$ is $\mathcal Z$-absorbing.
\end{thm}

The main motivation of this work is the converse question of this theorem.

\subsection{Uniform Rokhlin property and Cuntz comparison of open sets}

The following two properties were introduced in \cite{Niu-MD-Z}.
\begin{defn}[Definition 3.1 and Definition 4.1 of \cite{Niu-MD-Z}]\label{defn-URP-COS}
A topological dynamical system $(X, \Gamma)$, where $\Gamma$ is a discrete amenable group, is said to have the uniform Rokhlin property (URP) if for any $\eps>0$ and any finite set $K\subseteq \Gamma$, there exist closed sets $B_1, B_2, ..., B_S \subseteq X$ and $(K, \eps)$-invariant sets $\Gamma_1, \Gamma_2, ..., \Gamma_S \subseteq \Gamma$ such that the transformed sets 
$$B_s\gamma,\quad \gamma\in \Gamma_s,\quad s=1, ..., S, $$
are mutually disjoint and  
$$\mathrm{ocap}(X\setminus\bigsqcup_{s=1}^S\bigsqcup_{\gamma\in \Gamma_s}B_s\gamma) < \eps,$$
where the abbreviation $\mathrm{ocap}$ stands for orbit capacity (see, for instance, Definition 5.1 of \cite{Lindenstrauss-Weiss-MD}).

The dynamical system $(X, \Gamma)$ is said to have $(\lambda, m)$-Cuntz-comparison of open sets, where $\lambda\in (0, 1]$ and $m\in \mathbb N$, if for any open sets $E, F\subseteq X$ with $$ \mu(E) < \lambda \mu(F),\quad \mu \in\mathcal M_1(X, \Gamma),$$ 
where $\mathcal{M}_1(X, \Gamma)$ is the simplex of all invariant probability measures on $X$, it follows that 
$$\varphi_E \precsim \underbrace{\varphi_F\oplus\cdots \oplus \varphi_F}_m\quad\mathrm{in}\ \mathrm{C}(X)\rtimes\Gamma,$$ 
where $\varphi_E$ and $\varphi_F$ are continuous functions with open supports $E$ and $F$ respectively.

The dynamical system $(X, \Gamma)$ is said to have Cuntz comparison of open sets (COS) if it has $(\lambda, m)$-Cuntz-comparison on open sets for some $\lambda$ and $m$.
\end{defn}

\begin{thm}[\cite{Niu-MD-Z} and \cite{Niu-MD-Zd}]
Let $(X, \Gamma)$ be a minimal and free dynamical system.
\begin{itemize}
\item If $\Gamma = \Int^d$, then $(X, \Gamma)$ has the (URP) and (COS).
\item If $\Gamma$ is finitely generated and has sub-exponential growth, and if $(X, \Gamma)$ has a Cantor factor, then $(X, \Gamma)$ has the (URP) and (COS).
\end{itemize}
\end{thm}

The (UPR) implies that the C*-algebra $\mathrm{C}(X)\rtimes\Gamma$ can be weakly tracially approximated by the homogeneous C*-algebras generated by the Rokhlin towers. Together with the (COS), it has the following implications for the C*-algebra $\mathrm{C}(X) \rtimes \Gamma$:
\begin{thm}[\cite{Niu-MD-Z}, \cite{Niu-MD-Z-absorbing}, \cite{LN-sr1}]
Let $(X, \Gamma)$ be a minimal and free dynamical system with the (URP) and (COS), and let $A = \mathrm{C}(X) \rtimes\Gamma$. Then:
\begin{itemize}
\item $\mathrm{rc}(A) \leq \frac{1}{2} \mathrm{mdim}(X, \Gamma)$, where $\mathrm{rc}(A)$ is the radius of comparison of $A$; 
\item $A$ has stable rank one, i.e., invertible elements are dense;
\item $A \cong A\otimes\mathcal Z$ if, and only if, $A$ has strict comparison for positive elements; in particular,
\item if $(X, \Gamma)$ has the (SBP), then $A \cong A\otimes\mathcal Z$.
\end{itemize}

\end{thm}

\section{Property (C)}

\begin{defn}\label{defn-C}
Let $A$ be a unital C*-algebra and let $D$ be a unital commutative sub-C*-algebra. Then the pair $(D, A)$ is said to have Property (C) if for any positive contractions $f, g, h \in D$ satisfying $f, g \in \overline{hDh}$, and $$\mathrm{d}_\tau(f) < \mathrm{d}_\tau(g),\quad \tau \in \tr(A),$$ and for any $\eps>0$, there is a contraction $u \in \overline{hAh} + \Comp 1_A$ such that
$$ u f u^* \in_\eps^{\norm{\cdot}_2} \overline{gAg},$$
$$\mathrm{dist}_{2, \tr(A)}(udu^*, (D)_1) < \eps,\quad \mathrm{dist}_{2, \tr(A)}(u^*du, (D)_1) < \eps, \quad d \in (D)_1, $$ 
and
$$\norm{uu^* - 1}_{2, \tr(A)},\  \norm{u^*u - 1}_{2, \tr(A)} < \eps.$$
\end{defn}

\begin{rem}
Comparing to Property (COS), the approximations in Property (C) are with respect to the uniform trace norm, but on the other hand, the comparison in Property (C) is implemented by an almost unitary which is also an almost normalizer (with respect to the uniform trace norm).
\end{rem}

\begin{thm}\label{SBP=>C}
Let $(X, \Gamma)$ be a minimal and free dynamical system, where $\Gamma$ is a countable discrete amenable group.  Assume that $(X, \Gamma)$ has the (SBP), then $(D, A)$ has the Property (C).
\end{thm}

\begin{proof}
Let $\eps>0$ be arbitrary, and let $f, g, h \in D$ be positive contractions satisfying $f, g \in \overline{hDh}$, and $$\mathrm{d}_\tau(f) < \mathrm{d}_\tau(g),\quad \tau \in \tr(A).$$

Define the open sets
$$E_{(f-\eps)_+} = f^{-1}(\eps, +\infty),\quad E_g = g^{-1}(0, +\infty), \quad \mathrm{and} \quad E_h = h^{-1}(0, +\infty). $$

Then there are finite subset $ K \subseteq \Gamma$ and $\delta>0$ such that if a finite set $\Gamma_0 \subseteq \Gamma$ is $(K, \delta)$-invariant, then
$$ 
\frac{ \abs{x\Gamma_0 \cap E_{(f-\eps)_+}}}{\abs{\Gamma_0}} < \frac{ \abs{x\Gamma_0 \cap E_g}}{\abs{\Gamma_0}}, \quad x \in X.
$$

Since $(X, \Gamma)$ has the (SBP), by \cite{KS-comparison}, it has the (URP), and therefore there are mutually disjoint closed Rokhlin towers 
$$(B_1, \Gamma_1), ..., (B_S, \Gamma_S)$$
such that each $\Gamma_s$, $s=1, ..., S$, is $(K, \delta)$-invariant, and $$\mu(X \setminus \bigsqcup_{s=1}^S \bigsqcup_{\gamma\in \Gamma_s} B_s\gamma) < \eps,\quad \mu \in \mathcal M_1(X, \Gamma).$$

For each $x \in B_s$, $s=1, ..., S$,  since $\Gamma_s$ is $(K, \delta)$-invariant, one has
$$ \abs{x\Gamma_s \cap E_{(f-\eps)_+}} < \abs{x \Gamma_s \cap E_g}.$$ 
Note that $x\Gamma_s \cap E_{(f-\eps)_+}, x\Gamma_s \cap E_{g} \subseteq x\Gamma_s \cap E_{h}$. Therefore, there is a permutation $\sigma_x$ on $x\Gamma_s$ such that its restriction to $x\Gamma_s \cup (E_h)^c$ is identity and
$$\sigma_x(x\Gamma_s \cap E_{(f-\eps)_+}) \subseteq x\Gamma_s \cap E_{g}.$$ 
Since $E_g \subseteq X$ is open, there is a neighborhood $U_x \subseteq X$ such that 
$$ \gamma \in \Gamma_s\  \textrm{and}\  x\gamma \in E_{g}  \quad \Rightarrow \quad U_x \gamma \in E_{g},$$
and hence
$$\sigma_x(\{\gamma: \gamma \in \Gamma_s,\ x\gamma \in E_{(f-\eps)_+}\}) \subseteq \{\gamma \in \Gamma_s: U_x \gamma \subseteq E_{g}  \}$$
One also assume that $U_x$ is small enough so $(U_x, \Gamma_s)$ is an open tower.

Thus, the permutation unitary matrix $$u_{\sigma_x} \in \mathrm{M}_{\abs{\Gamma_s}}(\Comp) \subseteq (\mathrm{M}_{\abs{\Gamma_s}}(\mathrm{C_0}(U_x))) + \Comp 1$$ has the property that if $\phi_{x}$ is a continuous function supported in $U_x$, then 
\begin{equation}\label{into-u}
u_{\sigma_x} ( \sum_{\gamma \in \Gamma_s} (f - \eps)_+(x\gamma) \gamma(\phi_x) ) u^*_{\sigma_x} \subseteq \overline{gAg},
\end{equation}
where $\sum_{\gamma \in \Gamma_s} (f - \eps)_+(x\gamma) \gamma(\phi_x)$ is regarded as an element of $\mathrm{M}_{\abs{\Gamma_s}}(\mathrm{C_0}(U_x)))$, the C*-algebra of the open tower $(U_x, \Gamma_s)$.
Since $u_x$ is a permutation matrix, it preserves the diagonal subalgebra of $\mathrm{M}_{\abs{\Gamma_s}}(\mathrm{C}_0(U_x))$:
$$u_{\sigma_x} (\sum_{\gamma \in \Gamma_s} h \gamma(\phi_x) )u_{\sigma_x} \in \mathrm{C}(X),\quad h \in \mathrm{C}(X). $$

One also assume that $U_x$ is sufficiently small so $$\abs{(f - \eps)_+(y_1) - (f - \eps)_+(y_2)} < \eps,\quad y_1, y_2 \in U_x.$$

Since $B_s$ is compact, there are $x_{s, 1}, x_{s, 2}, ..., x_{s, n_s}$ and open sets $U_{s, 1} \ni x_{s, 1}$, ..., $U_{s, n_s} \ni x_{s, n_s}$ such that $\{U_{s, 1}, ..., U_{s, n_s}\}$ is an open cover of $B_s$. Since $(X, \Gamma)$ has (SBP), there are closed sets $x_{s, 1} \in V_{s, 1} \subseteq U_{s, 1}$, ..., $x_{s, n_s} \in V_{s, n_s} \subseteq U_{s, n_s}$ such that $V_{s, 1}, ..., V_{s, n}$ are mutually disjoint, and
$$\mu(B_s \setminus (V_1 \sqcup \cdots \sqcup V_n)) < \frac{\eps}{\abs{\Gamma_s}},\quad \mu \in \mathcal M_1(X, \Gamma).$$

Note that the tower $(B_s, \Gamma_s)$ splits into thiner towers $(V_{s, 1}, \Gamma_s)$, ..., $(V_{s, n_s}, \Gamma_s)$.

Choose continuous functions $$\phi_{x_i}: X \to [0, 1],\quad i=1, ..., n$$
such that
$$\phi_{x_{s, i}}|_{V_{s, i}} = 1 \quad \mathrm{and} \quad \phi_{x_{s, i}} |_{U_{s, i}^c} = 0,\quad i=1, ..., n,$$
and
$$ \gamma(\phi_{x_{s, i}}),\quad s=1, ..., S,\ i=1, ..., n,\ \gamma \in \Gamma_s$$
are mutually orthogonal.

Then
\begin{eqnarray*}
(f - \eps)_+ & \approx^{\norm{\cdot}_2}_{2 \eps} & (f- \eps)_+ \sum_{s=1}^S\sum_{i=1}^n \sum_{\gamma \in \Gamma_s} \gamma(\phi_{x_{s, i}}) \\
& \approx_\eps & \sum_{s = 1}^S \sum_{i=1}^n \sum_{\gamma \in \Gamma_s}  (f - \eps)_+(x_{s, i}\gamma) \gamma(\phi_{x_{s, i}}). 
\end{eqnarray*}

On each tower $(V_i, \Gamma_s)$, consider the permutation unitary $u_{\sigma_{x_i}} \in \mathrm{M}_{\abs{\Gamma_s}}(\Comp) $, and extend it to $(\tilde{V}_i, \Gamma_s)$, where each $\tilde{V}_i$, $i=1, ..., n$, is an open neighborhood of $V_i$ such that $$(\tilde{V}_i, \Gamma_s),\quad i=1, ..., n_s,\ s=1, ..., S$$ are mutually disjoint (open) towers.

Extend each permutation unitary $u_{x_{s, i}}$, $i=1, ..., n_s$, $s=1, ..., S$, to a contraction of $\mathrm{M}_{\abs{\Gamma_s}}(\mathrm{C}_0(\tilde{V}_{s, i}))$, the C*-algebra of the open tower $(\tilde{V}_{s, i}, \Gamma_s)$, which is regarded as a sub-C*-algebras of $\mathrm{C}(X) \rtimes \Gamma$. Note that, since the towers $(\tilde{V}_i, \Gamma_s),\quad i=1, ..., n_s,\ s=1, ..., S$, are mutually disjoint, the contractions $u_{x_{s, i}}$, $i=1, ..., n_s$, $s=1, ..., S$, are mutually orthogonal. Hence
$$u:= \sum_{s=1}^S\sum_{i=1}^{n_s} u_{x_{s, i}}$$
is a contraction.

Let us show that $u$ has the desired properties: By \eqref{into-u}, 
\begin{eqnarray*}
u(f - \eps)_+u^* & \approx_{2\eps}^{\norm{\cdot}_2} & u \sum_{s=1}^S\sum_{i=1}^{n_s}\sum_{\gamma \in \Gamma_s}(f - \eps)_+(x_{s, i} \gamma) \gamma(\phi_{s, i}) u^* \\
& = & \sum_{s=1}^S \sum_{i=1}^{n_s} u_{x_{s, i}}  (\sum_{\gamma \in \Gamma_s}(f - \eps)_+(x_{s, i} \gamma) \gamma(\phi_{s, i})) u^*_{x_{s, i}} \in \overline{gAg}.
\end{eqnarray*}

The element $u$ is also an almost normalizer: Indeed, for any contraction $c \in \mathrm{C}(X)$, one has
\begin{eqnarray*}
u c u^* & \approx_{2\eps}^{\norm{\cdot}_2} & u \sum_{s=1}^S\sum_{i=1}^{n_s} \sum_{\gamma \in \Gamma_s} c \gamma(\phi_{s, i}) u^* \\
& = & \sum_{s=1}^S \sum_{i=1}^{n_s} u_{x_{x_{s, i}}} ( \sum_{\gamma \in \Gamma_s}  c \gamma(\phi_{s, i})) u_{x_{s, i}}^* 
 \in  \mathrm{C}(X).
\end{eqnarray*}
Applying the equation above with $c = 1$ shows that
$$uu^* \approx_{4\eps}^{\norm{\cdot}_2} 1, $$
The similar argument shows that
$$u^*c u \in_{2\eps}^{\norm{\cdot}_2} \mathrm{C}(X), \quad c\in \mathrm{C}(X),\ \norm{h}\leq 1$$
and
$$u^*u \approx_{4\eps}^{\norm{\cdot}_2} 1,$$
as desired.
\end{proof}

\section{Property (S)}

Motivated by Theorem \ref{SBP-2-norm}, let us introduce the following property of a C*-algebra.
\begin{defn}\label{defn-S}
Let $A$ be a unital C*-algebra, and let $\Delta \subseteq \tr(A)$ be a closed set of tracial states. The pair $(A, \Delta)$ will be said to have Property (S) if for any self-adjoint contraction $f$ and any $\eps>0$, there is a self-adjoint contraction $g \in A$ such that 
\begin{enumerate}
\item $\norm{f - g }_{2, \Delta} < \eps$, and
\item there is $\delta>0$ such that $\tau(\chi_\delta(g)) < \eps$, $\tau \in \Delta$,
where 
\begin{equation}\label{chi-defn}
\chi_\delta(t) = \left\{ \begin{array}{ll}
1, & \abs{t} < \delta, \\
2-\abs{t}/\delta, & \delta\leq \abs{t} < 2\delta, \\
0, & \mathrm{otherwise}.
\end{array}\right.
\end{equation}
\end{enumerate}

In the case that $\Delta = \tr(A)$, we shall just say that $A$ has Property (S) if $(A, \tr(A))$ has Property (S).
\end{defn}

Compared to Theorem \ref{SBP-2-norm}, Property (S) can be regarded as a weaker version of the small boundary property, without referring to a commutative subalgebra $D$. 

In general, it is weaker than the actual small boundary property, as shown in Example \ref{weak-S} later in this section. However, it will be shown  (Theorem  \ref{prop-S}) that, if some other properties (Properties (C) and (E)) are provided, then Property (S) for $A$ implies the small boundary property of the pair $(D, A)$.

Property (S) is closely related to the real rank of the sequence algebra:
\begin{defn}\label{qRR-defn}
Let $A$ be a C*-algebra, and let $\Delta\subseteq\tr(A)$. Let us say that $\mathrm{qRR}(l^\infty(A)/J_{2, \omega, \Delta}) = 0$ if the class of the constant sequence of any self-adjoint contraction of $A$ can be approximated by invertible self-adjoint contractions of $l^\infty(A)/J_{2, \omega, \Delta}$ with respect to the uniform limit trace norm $\norm{\cdot}_{2, \omega, \Delta}$. It is clear that if $\mathrm{RR}(A) = 0$ or if $\mathrm{RR}(l^\infty(A)/J_{2, \omega, \Delta}) = 0$, then $\mathrm{qRR}(l^\infty(A)/J_{2, \omega, \Delta}) = 0$.

\end{defn}

\begin{prop}\label{character-S}
Let $A$ be a unital C*-algebra, and let $\Delta\subseteq\tr(A)$ be closed. The following conditions are equivalent:
\begin{enumerate}
\item[(1)] $\mathrm{qRR}(l^\infty(A)/J_{2, \omega, \Delta}) = 0$;
\item[(2)] $(A, \Delta)$ has Property (S);
\item[(3)] $\mathrm{RR}(l^\infty(A)/J_{2, \omega, \Delta}) = 0$.
\end{enumerate}
\end{prop}

\begin{proof}
$(1) \Rightarrow (2)$: Assume $\mathrm{qRR}(l^\infty(A)/J_{2, \omega, \Delta}) = 0$.  Let $(f, \eps)$ be given, where $f \in A$ is a self-adjoint contraction and $\eps>0$. Consider the constant sequence $(f)$, and then there is a sequence of self-adjoint contractions $(g_k) \in l^\infty(A)$ such that 
$$\norm{(f - g_k)}_{2, \omega, \Delta} = \lim_{k \to \omega}\norm{f - g_k}_{2, \Delta} < \eps$$
and $\overline{(g_k)}$ is invertible in $l^\infty(A)/J_{2, \omega, \Delta}$. Then there is $\delta>0$ such that
$$\overline{(\chi_\delta(g_k))} = \chi_\delta(\overline{(g_k)}) = 0.$$ 
In other words,
$$ (\chi_\delta(g_k)) \in J_{2, \omega, \Delta}.$$
Then, with some sufficiently large $k$, one has
\begin{itemize}
\item $\norm{f - g_k}_{2, \Delta} < \eps$, and
\item $\tau(\chi_\delta(g_k)) < \eps$, $\tau\in \Delta$,
\end{itemize}
as desired.

$(2) \Rightarrow (3)$: Assume that $(A, \Delta)$ has Property (S), and let us show that $l^\infty(A)/ J_{2, \omega, \Delta}$ has real rank zero. Let $\overline{(a_1, a_2, ...)} \in l^\infty(A)/ J_{2, \omega, \Delta}$ be a self-adjoint contraction (without loss of generality, one may assume that each $a_n$, $n=1, 2, ...$, is a contraction), and let $\eps>0$ be arbitrary. By Property (S), for each $n=1, 2, ...$, there is a self-adjoint contraction $b_n$ such that
\begin{itemize}
\item $\tau((a_n - b_n)^2) < 1/n$ for all $\tau\in\Delta$, and
\item there is $\delta_n>0$ such that $\tau(\chi_{\delta_n}(b_n)) < 1/n$ for all $\tau\in\Delta$.
\end{itemize}
Without loss of generality, one may assume that $\delta_n<\eps$.
Note that $$\overline{(a_1, a_2, ...)} = \overline{(b_1, b_2, ...)}.$$

Consider $b_n' = g_n(b_n)$ and $c_n = h_n(b_n)$, where
$$
g_n(t) = \left\{ \begin{array}{ll}
\eps t/\delta_n, & \abs{t} < \delta_n, \\
t + \mathrm{sign}(t)(\eps-\delta_n), & \mathrm{otherwise}
\end{array}\right.
$$
and
$$
h_n(t) = \left\{ \begin{array}{ll}
t/\eps\delta_n, & \abs{t} < \delta_n, \\
1/(t + \mathrm{sign}(t)(\eps-\delta_n)), & \mathrm{otherwise}.
\end{array}\right.
$$
Then 
$$\norm{b_n' - b_n} < \eps,\quad \norm{c_n} < 1/\eps,\quad \mathrm{and}\quad \tau((b_n'c_n - 1)^2) <\tau(\chi_{\delta_n}(b_n))< 1/n,\quad \tau \in \Delta.$$ In particular, the element $\overline{(b'_1, b'_2, ...)}$ is invertible in $l^\infty(A)/J_{\omega, \Delta}$ with the inverse $\overline{(c_1, c_2, ...)}$. Moreover, $$\norm{\overline{(a_1, a_2, ...)} - \overline{(b'_1, b'_2, ...)}} = \norm{\overline{(b_1, b_2, ...)} - \overline{(b'_1, b'_2, ...)}}< \eps.$$ This shows that $l^\infty(A)/ J_{2, \omega, \Delta}$ has real rank zero.

$(3) \Rightarrow (1)$: Trivial.
\end{proof}

Uniform property $\Gamma$ implies Property (S):
\begin{prop}\label{Property-S-Prop}
If a unital C*-algebra $A$ has uniform property $\Gamma$, then $A$ has Property (S).
\end{prop}

\begin{proof}
The proof is similar to the proof of Theorem 3.5 of \cite{EN-SBP-I}:

Let $(f, \eps)$ be given, where $f \in A$ is a self-adjoint contraction and $\eps>0$. By Corollary 3.9 of \cite{EN-SBP-I}, 
for the given $\eps$, there exist $n\in\mathbb N$ and self-adjoint contractions $$f_1, f_2, ..., f_n \in A$$  such that
\begin{equation}\label{mult-pert-eq-1}
\norm{f - f_i} < \eps,\quad i=1, 2, ..., n, 
\end{equation}
and 
there is $\delta>0$ such that 
\begin{equation*}
\frac{1}{n}(\tau((f_1)_\delta) + \cdots + \tau((f_n)_\delta)) < \eps,\quad \tau\in\mathrm{T}(A),
\end{equation*} 
where $(f)_\delta = \chi_\delta(f)$ (see \eqref{chi-defn}). 
In particular, regarding $(f_1)_\delta, ..., (f_n)_\delta$ as constant sequences in $A$, we have
\begin{equation}\label{mult-pert-eq-2}
\frac{1}{n}(\tau((f_1)_\delta) + \cdots + \tau((f_n)_\delta)) < \eps,\quad \tau\in\tr(A)_\omega.
\end{equation} 

Since $A$ has uniform property $\Gamma$, there is a partition of unity $$p_1, p_2, ..., p_n \in ( l^\infty(A)/J_{2, \omega, \Delta}) \cap A' $$
such that
\begin{equation}\label{r-gamma-cut-eq}
\tau(p_iap_i) = \frac{1}{n} \tau(a),\quad a\in A,\ \tau\in \tr(A)_\omega.
\end{equation}

Consider the contraction $$g:=p_1\overline{(f_1)}p_1 + \cdots + p_n\overline{(f_n)}p_n \in l^\infty(A)/J_{2, \omega, \Delta}. $$
By \eqref{mult-pert-eq-1},
\begin{equation}\label{close-cond-1}
\norm{f-g} = \norm{p_1\overline{(f-f_1)}p_1 + \cdots + p_n\overline{(f-f_n)}p_n}  < \eps.
\end{equation}
Note that, for each $\tau \in \tr(A)_\omega$, by \eqref{r-gamma-cut-eq},
$$\tau(p_i\overline{((f_i)_\delta)} p_i) = \frac{1}{n}\tau((f_i)_\delta),\quad i=1, ..., n, \ \tau\in\tr(A)_\omega,$$
and hence, together with \eqref{mult-pert-eq-2},
\begin{eqnarray}\label{close-cond-2}
\tau((g)_\delta) & = & \tau(p_1\overline{((f_1)_\delta)}p_1) + \cdots + \tau(p_n\overline{((f_n)_\delta)}p_n) \\
& = & \frac{1}{n}(\tau((f_1)_\delta) + \cdots + \tau((f_n)_\delta)  ) \nonumber \\
& < & \eps.  \nonumber
\end{eqnarray}
Pick a representative sequence $g = \overline{(g_k)}$ with $g_k$, $k=1, 2, ...$, self-adjoint contractions of $A$. By \eqref{close-cond-1} and \eqref{close-cond-2}, with some sufficiently large $k$, the function $g_k$ satisfies
\begin{enumerate}
\item $\norm{f - g_k}_{2, \tr(A)} < \eps$, and
\item $\tau((g_k)_\delta) < \eps$, $\tau\in \tr(A)$,
\end{enumerate}
as desired.
\end{proof}

But Property (S) in general is a weaker property than the actual small boundary property:
\begin{example}\label{weak-S}
Let
$$
\xymatrix{
A_1 \ar[r] & A_2 \ar[r] & \cdots \ar[r] & A}
$$
be a unital AH algebra such that $A_s = \mathrm{M}_{n_s}(\mathrm{C}([0, 1]^{d_s}))$, and $$ \lim_{s \to\infty}n_s = \infty \quad \mathrm{and}\quad  \limsup_{s \to \infty}\frac{d_s}{n_s} = \gamma < \infty.$$ Such C*-algebras include Villadsen algebras which are not $\mathcal Z$-absorbing (so the standard commutative sub-C*-algebra does not have the small boundary property).

Note that $\mathrm{M}_n(\Comp)^{s.a.}$ is a manifold with of dimension $n^2$ and $\mathrm{U}_n(\Comp)$ is a compact Lie group of dimension $n^2$.

Denote by $V_{n, k}$ the set of self-adjoint $n \times n$ matrices with at least $k$ eigenvalues which are zero. It can be written as
$$ V_{n, k}=\{u\mathrm{diag}\{\underbrace{0, ..., 0}_k, \lambda_1, ..., \lambda_{n-k}\}u^*: u \in \mathrm{U}_n(\Comp),\ \lambda_1, ..., \lambda_{n-k} \in \Real\},$$
which is the $\mathrm{U}_n(\Comp)$-orbit of the following $(n-k)$ dimensional submanifold: $$\{\mathrm{diag}\{\underbrace{0, ..., 0}_k, \lambda_1, ..., \lambda_{n-k}\},\quad \lambda_1, ..., \lambda_{n-k} \in \Real\} \subseteq \mathrm{M}_n(\Comp)^{s.a.}.$$ 
For each diagonal matrix above, its stabilizer contains the subgroup
$$
\left\{
\left( 
\begin{array}{cc}
U & 0 \\
0 & I_{n-k}
\end{array}
\right): U \in \mathrm{U}_k(\Comp)
\right\} \subseteq \mathrm{U}_n(\Comp),
$$
which has dimension $k^2$.
The set $V_{n, k}$ is a union of finitely many submanifolds of $ \mathrm{M}_n(\Comp)^{s.a.}$ with dimension at most $$ (n - k) + n^2 - k^2.$$

Now, let $\eps>0$ be arbitrary, and let $a \in A_1$ be a self-adjoint element.
Choose $s$ large enough that if $A_s$ is written as $\mathrm{M}_{n}(\mathrm{C}([0, 1]^{d}))$, then
$$ \frac{d}{n} \leq \gamma+1$$
and
$$
(1 - \eps^2) n^2 + 2\eps n +  (1-\eps)n + (\gamma+1) n + 1 \leq n^2.
$$
Choose a natural number $k$ such that
$$
\frac{k}{n} \leq \eps < \frac{k+1}{n}.
$$

Note that the set $V_{n, k}$ has dimension at most
$$(n-k) + n^2 - k^2 \leq (n - \eps n +1) + n^2 - (\eps n-1)^2.$$
Then 
\begin{eqnarray*}
\mathrm{dim}(V_{n, k}) + d + 1 & \leq & (n - \eps n +1) + (n^2 - (\eps n-1)^2) + (\gamma + 1) n + 1 \\
& = & (1-\eps^2)n^2 +2\eps n  + (1-\eps)n + (\gamma+ 1)n+ 1 \\
&\leq & n^2 = \mathrm{dim}(\mathrm{M}_n(\Comp)^{s.a.}).
\end{eqnarray*}

Hence, the self-adjoint element $a$, regarded as a continuous map from $[0, 1]^d$ to $\mathrm{M}_n(\Comp)^{s.a.}$, can be approximated by continuous maps $b$ from $[0, 1]^d$ to $\mathrm{M}_n(\Comp)^{s.a.} \setminus V_{n, k}$, and so at most $k$ of the eigenvalues of $b(x)$ can be zero at every $x \in [0, 1]^d$, and hence there is $\delta_x>0$ such that at most $k$ of the eigenvalues of $b(x)$ are in the $\delta_x$-neighbourhood of $0$. Then, by  the continuity of the function $b$, a compactness argument shows that there is $\delta>0$ such that for every $x\in [0, 1]^d$, there at most $k$ of the eigenvalues of $b(x)$ are in the $\delta$-neighbourhood of $0$, and hence
$$ \tau(\chi_{\delta/2}(b)) < \frac{k}{n} < \eps,\quad \tau \in \tr(A_s).$$
So the C*-algebra $A$ has Property (S).

\end{example}

In the recent work \cite{Vaccaro:2026aa}, the example above is generalized to $\mathrm{C}(X)\rtimes \Gamma$:
\begin{thm}[\cite{Vaccaro:2026aa}]\label{S-holds}
Let $A \mathrm{C}(X) \rtimes \Gamma $, where $(X, \Gamma)$ is a free and minimal dynamical system with the (URP), or let $A$ be an AH algebra with diagonal maps. Then $A$ has the Property (S).
\end{thm}

\section{An existence property}

The following definition is an existence property for a pair of C*-algebras $D \subseteq A$.

\begin{defn}\label{defn-D}
Let $A$ be a unital C*-algebra and let $D$ be a unital commutative subalgebra. 

The pair $(D, A)$ is said to have Property (E) if for any positive contraction $a \in A$, its the trace spectral distribution can be approximated by the trace spectral distribution of the positive contractions of $D$ in the sense that for any finite subset $\mathcal F \subseteq \mathrm{C}([0, 1])$, and any $\eps>0$, there is a positive contraction $b \in D$ such that
$$ \abs{ \tau(f(a)) - \tau(f(b))} < \eps,\quad f \in \mathcal F,\ \tau \in \tr(A).$$
\end{defn}

This property holds for all AH algebras with diagonal maps and all crossed products $\mathrm{C}(X) \rtimes \Gamma$ where $(X, \Gamma)$ is a minimal and free dynamical system with the (URP). Recall the following theorem which is due to Thomsen and Li (\cite{Thomsen-AI-Tr} and \cite{Li-interval}).
\begin{thm}\label{Thomsen-Li}
Suppose that $X$ is a path connected metrizable compact space. For any finite subset $\mathcal F \subseteq \mathrm{Aff}\mathrm{T}(\mathrm{C}(X))$ and $\eps>0$, there is $N > 0$ with the following property:

For any unital positive linear map $\xi: \mathrm{Aff}\mathrm{T}(\mathrm{C}(X)) \to \mathrm{Aff}\mathrm{T}(\mathrm{C}(Y))$, where $Y$ is an arbitrary metrizable compact space, and any $n\geq N$, there are homomorphisms 
$$ \phi_1, ..., \phi_n: \mathrm{C}(X) \to  \mathrm{C}(Y) $$
such that
$$\abs{\xi(f)(\tau) -  \frac{1}{n} \sum_{i=1}^n \tau(\phi_i(f)) } < \eps,\quad f \in \mathcal F,\ \tau \in \mathrm{T}(\mathrm{C}(Y)). $$
\end{thm}

With Property (URP), the crossed product C*-algebra $\mathrm{C}(X) \rtimes \Gamma$ has the following tracial approximation property:
\begin{lem}[cf.~Theorem 3.9 of \cite{Niu-MD-Z}]\label{pre-approx}
Let $(X, \Gamma)$ be a free and minimal dynamical system with the (URP). Then, for any finite set $ \mathcal F \subseteq \mathrm{C}(X)$ and any $\eps>0$, there exist a positive element $p \in \mathrm{C}(X)$ with $\norm{p} = 1$ and a sub-C*-algebra $C \subseteq \mathrm{C}(X) \rtimes\Gamma$ such that
\begin{enumerate}
\item $ \norm{ [p, f] } < \eps$, $ f \in \mathcal F$,
\item $pfp \in_\eps C $, $f \in \mathcal F$, 
\item $pdp \in C$, $d \in \mathrm{C}(X)$,  
\item $C \cong \bigoplus_{i=1}^S \mathrm{M}_{n_i}(\mathrm{C}_0(Z_i))$, where $Z_i \subseteq X$, $i=1, ..., S$, are mutually disjoint, and under this isomorphism, the elements $pdp$ are  diagonal elements of $C$ for all $d \in \mathrm{C}(X)$,
\item under the isomorphism above, all diagonal elements of $\bigoplus_{i=1}^S \mathrm{M}_{n_i}(\mathrm{C}_0(Z_i))$ are in $C \cap \mathrm{C}(X)$,
\item $\mathrm{d}_\tau(1-p) < \eps$, $\tau \in \tr(A)$, 
\item there is a closed subset $[Z_{i}] \subseteq Z_i$ for each $i=1, ..., S$ such that if $a \in C$, $\norm{a} \leq 1$, and $a$ is supported in $\bigsqcup_{i=1}^S (Z_i \setminus [Z_{i}])$ under the isomorphism $C \cong \bigoplus_{i=1}^S\mathrm{M}_{n_i}(\mathrm{C}_0(Z_i))$, then $$ \tau(a) < \eps, \quad \tau \in \tr(A),$$
\item for any $d \in \mathrm{C}(X)$, there are $d_1, d_2 \in \mathrm{C}(X)$ such that  $$d = d_1 + d_2,\quad \norm{d_1}_{2, \tr(A)} < \eps,\quad d_2 \in C \cap \mathrm{C}(X), $$ and the support of $d_2$ is inside $\bigsqcup_{i=1}^S [Z_i]$ under the isomorphism $C \cong \bigoplus_{i=1}^S \mathrm{M}_{n_i}(\mathrm{C}_0(Z_i))$.
\end{enumerate}
\end{lem}
\begin{proof}
This follows from the same argument as Theorem 3.9 of \cite{Niu-MD-Z} (or, actually a simpler one).
\end{proof}

\begin{thm}\label{Property-D}
Let $A$ be a simple AH algebra with diagonal maps, or let $A = \mathrm{C}(X) \rtimes \Gamma$, where $(X, \Gamma)$ is a free and minimal dynamical system with the (URP). Then the pair $(D, A)$ has Property (E).
\end{thm}

\begin{proof}
Let $(\mathcal F, \eps)$ be given. Without loss of generality, one may assume that each element of $\mathcal F$ has norm $1$ and is real valued,  so $\mathcal F$ can be regarded as a subset of $\mathrm{Aff}(\tr(\mathrm{C}([0, 1])))$.  Applying the Thomsen-Li Theorem above to $(\mathcal F, \eps)$ (where $X = [0, 1]$), one obtains $N$.

Let us first consider the case of AH algebras. Let $a \in A$ be a positive element with norm $1$; to prove the theorem, without loss of generality, one may assume that $a \in \mathrm{M}_n(\mathrm{C}(X)) \subseteq A$ for some $n \geq N$. Consider the homomorphism $\phi: \mathrm{C}[0, 1] \to \mathrm{M}_n(\mathrm{C}(X))$ induced by $a$, and consider the induced unital positive linear map $\phi^*: \mathrm{Aff}( \tr(\mathrm{C}[0, 1])) \to \mathrm{Aff}(\tr(\mathrm{C}(X)))$. By the Thomsen-Li Theorem, there are continuous maps $\lambda_1, ..., \lambda_n: X \to [0, 1]$ such that 
$$ \abs{\phi^*(f)(\tau) - \frac{1}{n}(\tau(f \circ \lambda_1) + \cdots + \tau(f \circ \lambda_n))} < \eps,\quad f\in \mathcal F,\  \tau \in \tr(\mathrm{C}(X)).$$
Then, with $$b = \mathrm{diag}\{(\lambda_1)_*(\mathrm{id}) , ..., (\lambda_n)_*(\mathrm{id})\} \in D_n,$$
(and since $\phi(f) = f(a)$,) one has 
$$ \abs{\tau(f(a)) - \tau(f(b))} < \eps,\quad  f\in \mathcal F,\ \tau \in \tr(\mathrm{M}_n(\mathrm{C}(X))).$$

Let us now consider the case that $A = \mathrm{C}(X) \rtimes \Gamma$, where $(X, \Gamma)$ is free and minimal, and has the (URP). 

Let $a \in A$ be a positive element with norm $1$. Choose $\delta>0$ such that if $\norm{x - y}_{2, \tr(A)} < \delta$, where $x, y$ are positive contractions, then $\norm{f(x) - f(y)}_{2, \tr(A)} < \eps$  for all $f \in \mathcal F$.

By Lemma \ref{pre-approx}, there exist a sub-C*-algebra $C \subseteq A$ and a positive contraction $p \in C$ such that
\begin{enumerate}
\item $C \cong \bigoplus_{i=1}^S \mathrm{M}_{n_i}(\mathrm{C}_0(Z_i))$, where $n_i>N$, $i=1, ..., S$,
\item $pap \in_{\delta/2} C$,
\item $\mathrm{d}_\tau(1-p) < \delta/2$, $\tau \in \tr(A)$,
\item\label{E-condition-4} there is a closed subset $[Z_{i}] \subseteq Z_i$ for each $i=1, ..., S$ such that if $a \in C$, $\norm{a} \leq 1$, and $a$ is supported in $\bigsqcup_{i=1}^S (Z_i \setminus [Z_{i}])$ under the isomorphism $C \cong \bigoplus_{i=1}^S\mathrm{M}_{n_i}(\mathrm{C}_0(Z_i))$, then $$ \tau(a) < \eps, \quad \tau \in \tr(A).$$
\end{enumerate}
Choose $\tilde{a} \in C$ such that $\norm{\tilde{a} - pap} < \delta/2$; hence
$$ \norm{a - \tilde{a}}_{2, \tr(A)} < \delta. $$ 
and so, by the choice of $\delta$, 
\begin{equation}\label{E-equation-1}
 \norm{f(a) - f(\tilde{a})}_{2, \tr(A)} < \eps,\quad f \in \mathcal F. 
 \end{equation}
Consider the homomorphism
$$ \phi: \mathrm{C}([0, 1]) \ni f \mapsto \pi_{[Z]}(f(\tilde{a})) \in \pi_{[Z]}(C) \cong  \bigoplus_{i=1}^S \mathrm{M}_{n_i}(\mathrm{C}([Z_i])),$$
where $[Z] = \bigsqcup_{i=1}^S [Z_i]$, and consider the unital positive linear map
$\phi^*:  \mathrm{Aff}( \tr(\mathrm{C}[0, 1])) \to \mathrm{Aff}(\tr(\mathrm{C}([Z]))).$ Then, by Theorem \ref{Thomsen-Li}, there are continuous maps $\lambda_{i, 1}, ..., \lambda_{i, n_i}: [Z_i] \to [0, 1]$, $i=1, ..., S$, such that
$$\abs{\phi^*(f)(\tau) - \frac{1}{n_i}(\tau(f \circ \lambda_{i, 1}) + \cdots + \tau(f \circ \lambda_{i, n_i}) )} < \eps,\quad f \in \mathcal F,\ \tau \in \tr(\mathrm{C}([Z_i])).$$
Define $$b_i = \mathrm{diag}\{ \mathrm{id} \circ \lambda_{i, 1}, ...,  \mathrm{id} \circ \lambda_{i, n} \} \in \mathrm{M}_{n_i}(\mathrm{C}([Z_i])),$$
where $\mathrm{id}$ is the identity function on $[0, 1]$,
and define 
$$b = b_1 \oplus \cdots \oplus b_{n_S}. $$ 
Then
\begin{equation}\label{E-equation-2}
\abs{\tau(\pi_{Z_0}(f(\tilde{a}))) - \tau(f(b))} < \eps,\quad f \in \mathcal F,\   \tau \in \tr(\bigoplus_{i=1}^S\mathrm{M}_{n_i}(\mathrm{C}([Z_i]))).
\end{equation}
Note that $b$ is a diagonal element. Extend $b$ to a positive diagonal contraction $\hat{b} \in C$. Note that, by Lemma \ref{pre-approx}(5), the element $\hat{b}$ is also in $D$. 
By \eqref{E-condition-4}, a calculation shows 
\begin{equation}\label{E-equation-3}
 \abs{\tau(f(\hat{b})) - \tau|_{[Z]}(f(b))} < 4 \eps,\quad f\in \mathcal F,\  \tau \in \tr(A),
 \end{equation}
where $\tau|_{[Z]}$ is the tracial state of $\bigoplus_{i=1}^S \mathrm{M}_{n_i}(\mathrm{C}([Z_i]))$ 
induced by $\tau|_C$:
$$\tau|_{[Z]}(x) = \lim_{n \to\infty}\frac{1}{\tau(e^2_n)} \tau(e_n \tilde{x} e_n), 
$$
where $x \in \bigoplus_{i=1}^S \mathrm{M}_{n_i}(\mathrm{C}([Z_i]))$, $\tilde{x} \in C \cong  \bigoplus_{i=1}^S \mathrm{M}_{n_i}(\mathrm{C}_0(Z_i))$ is an extension of $x$, and $(e_n)$ is a decreasing sequence of positive contractions of $C$ which converges (pointwise) to $\mathbf{1}_{Z}$. Also note that, by \eqref{E-condition-4},
\begin{equation}\label{E-equation-4}
\abs{\tau(f(\tilde{a})) - \tau|_{[Z]}(\pi_{[Z]}(f(\tilde{a})))} < 2\eps,\quad f \in \mathcal F,\  \tau \in \tr(A).
\end{equation}
Then, for all $f \in \mathcal F$ and $\tau \in \tr(A)$, by \eqref{E-equation-1}, \eqref{E-equation-4}, \eqref{E-equation-2}, and \eqref{E-equation-3}, 
$$ \tau(f(a)) \approx_\eps \tau(f(\tilde{a})) \approx_{2\eps} \tau|_{[Z]}(\pi_{Z_0}(f(\tilde{a}))) \approx_{\eps} \tau|_{[Z]}(f(b)) \approx_{4\eps} \tau(f(\hat{b})),$$
as desired. 
\end{proof}

\section{Some approximation lemmas}
In the section, let us prepare some approximation lemmas for the next section.

\begin{lem}\label{tr-to-norm-red}
Let $X$ be a metrizable compact space, and let $\Delta$ be a compact set of probability Borel measures of $X$. Then, for any $\eps \in (0, 1)$, there is $\delta>0$ such that if $a_0, a_1, d: X \to [0, 1]$ are continuous functions satisfying
\begin{enumerate}
\item $a_0a_1 = a_1$,
\item $\norm{a_0 d - d}_{2, \Delta} < \delta$, and
\item $\norm{da_1 - a_1}_{2, \Delta} < \delta$,
\end{enumerate}
then there is a continuous function $\tilde{d}: X \to [0, 1]$ such that
\begin{enumerate}
\item $\norm{d - \tilde{d}}_{2, \Delta} < \eps$,
\item $\norm{a_0 \tilde{d} - \tilde{d}} < \eps$, and
\item $\norm{\tilde{d} a_1 - a_1} < \eps$.
\end{enumerate}
\end{lem}

\begin{proof}
With the given $\eps$, choose $\eps' \in (0, 1)$ such that $\eps' < \eps$ and $2\sqrt{\eps'}< \eps$.
Then $$\delta= \eps'/\sqrt{2/(\eps')^3} 
$$
has the property of the lemma.

Indeed, let $a_0, a_1, d$ satisfy the conditions of the statement. Define sets 
$$A_{0, \leq 1-\eps'} = a_0^{-1}([0, 1-\eps']) \quad \mathrm{and}\quad A_{1, \geq \eps'} = a_1^{-1}([\eps', 1]).$$ Note that $A_{0, \leq 1-\eps'} \cap A_{1, \geq \eps'} = \O$.

On $A_{0, \leq 1-\eps'}$, consider the set
$$W_0 = \{x \in A_{0, \leq 1-\eps'}: \abs{d(x)}^2 \geq \eps'\}.$$ Then, for any $\mu \in \Delta$,
$$
(\eps')^2 \eps \mu(W_0) \leq (\eps')^2 \int_{A_{0, \leq 1-\eps'}} d^2(x) \mathrm{d}\mu(x) \leq \int_{A_{0, \leq 1-\eps'}} (a_0(x) - 1)^2d^2(x) \mathrm{d}\mu(x) < \delta^2,
$$
and hence
$$\mu(W_0) < \delta^2/(\eps')^3.$$

On $A_{1, \geq \eps'}$, consider the set
$$W_1 = \{x \in A_{1, \geq \eps'}: \abs{d(x) - 1}^2 \geq \eps'\}.$$ Then, for any $\mu \in \Delta$,
$$
(\eps')^2(\eps' \mu(W_1)) \leq (\eps')^2 \int_{A_{1, \geq \eps'}}\abs{d(x) - 1}^2 \mathrm{d}\mu(x) \leq \int_{A_{1, \geq \eps'}}\abs{d(x) - 1}^2 a_1^2(x) \mathrm{d}\mu(x) < \delta^2,
$$
and hence
$$\mu(W_1) < \delta^2/(\eps')^3.$$

Choose disjoint open sets $U_0 \supseteq W_0$ and $U_1 \supseteq$. Since $\Delta$ is compact, $U_0$ and $U_1$ can be chosen such that
$$ \mu(U_0) < \delta^2/(\eps')^3 \quad \mathrm{and} \quad \mu(U_1) < \delta^2/(\eps')^3,\quad \mu \in \Delta.$$ Choose continuous functions $r_0, r_1: X \to [0, 1]$ such that
$$r_i|_{W_i} = 1 \quad \mathrm{and} \quad r_i|_{X \setminus U_i} = 0,\quad i=0, 1.$$

Define 
$$
\tilde{d}(x) = \left\{
\begin{array}{ll}
0, & x \in W_0,\\
(1-r_0(x))d(x), & x \in U_0, \\
d(x), & x \in X \setminus (U_0 \cup U_1), \\
(1-r_1(x))d(x) + r_1(x), & x \in U_1, \\
1, & x \in W_1.
\end{array}
\right. 
$$

Then 
$$\int_X(d(x) - \tilde{d}(x))^2 \mathrm{d}\mu(x) \leq \mu(U_0) + \mu(U_1) < 2\delta^2/(\eps')^3,\quad \mu \in \Delta.$$ That is,
$$ \norm{d - \tilde{d}}_{2, \Delta} < \delta \sqrt{2/(\eps')^3} = \eps' < \eps.$$

Note that
$$\tilde{d}(x) < \sqrt{\eps'},\quad x \in A_{0, \leq 1-\eps'}.$$
So
$$\norm{a_0 \tilde{d} - \tilde{d}} < \max\{\eps', 2\sqrt{\eps'}\} < 2\sqrt{\eps'} < \eps.$$

Also note that
$$1-\tilde{d}(x) < \sqrt{\eps'}, \quad x \in A_{1, \geq \eps'}.$$
So
$$\norm{\tilde{d} a_1 - a_1} < \max\{\sqrt{\eps'}, 2\eps'\} < 2\sqrt{\eps'} < \eps,$$
as desired.
\end{proof}

The following lemma certainly is well known:
\begin{lem}\label{pert-almost-ortho}
Let $A$ be a C*-algebra. Let $N \in \mathbb N$ and $\eps>0$. Then there is $\delta>0$ such that if $c_1, ..., c_N$ are self-adjoint contractions such that $$ \norm{c_i c_j} < \delta,\quad i, j =1, ..., N,\ i\neq j,$$
then
$$\norm{c_1 + \cdots + c_N} < \max\{\norm{c_i}: i=1, ..., N\} + \eps.$$
\end{lem}

\begin{proof}
For the given $(N, \eps)$, there is $\delta>0$ such that if $c_1, ..., c_N$ are self-adjoint contractions such that $$ \norm{c_i c_j} < \delta,\quad i, j =1, ..., N,\ i\neq j,$$
then there are self-adjoint elements $\tilde{c}_1, ..., \tilde{c}_N \in A$ such that
$$ \norm{\tilde{c}_i - c_i} < \eps/2N, \quad \textrm{and} \quad \tilde{c}_i \perp \tilde{c}_j,\quad i, j = 1, ..., N,\  i\neq j.$$
Then, this $\delta$ has the property of the lemma, as
\begin{eqnarray*}
\norm{c_1 + \cdots + c_N} & \approx_{\eps/2} & \norm{\tilde{c}_1 + \cdots + \tilde{c}_N} \\
& = & \max\{\norm{\tilde{c}_i}: i=1, ..., N\} \\
& \approx_{\eps/2} & \max\{\norm{c_i}: i=1, ..., N\},
\end{eqnarray*}
as desired.
\end{proof}

\begin{lem}[cf.~\cite{CE-str1}]\label{hereditary-sets-norm-app}
Let $A$ be a C*-algebra. For any $\eps>0$, there are $N \in \mathbb N$ and $\delta>0$ with the following property:

Let $a \in A$ be a positive element with norm at most $1$. Define $$a_i = \chi_i(a),\quad  i=1, ..., N, $$ where $$\chi_i(t) = \left\{ 
\begin{array}{ll}
0, & t \leq \frac{i-1}{N}, \\
\mathrm{linear}, & t \in [\frac{i-1}{N}, \frac{i}{N}], \\
1, & t \geq \frac{i}{N}.
\end{array} 
\right.
$$
Assume there are positive elements $d_1, ..., d_N\in A$  with norm at most $1$ such that
         \begin{enumerate}
         \item\label{commut-cond} $[a, d_i] = 0$, $i=1, ..., N$,
         \item\label{here-cond-1-a} $ \norm{a_{i} d_{i+1} - d_{i+1}} < \delta$, $i=1, ..., N-1$,
         \item\label{here-cond-1-b} $ \norm{d_i a_{i+1} - a_{i+1}} < \delta$, $i=1, ..., N-1$.
         \end{enumerate} 
        Then
        $$\norm{ a  - \frac{1}{N}(d_1 + \cdots + d_N)} <  \eps.$$
\end{lem}

\begin{proof}
Choose $N > {32}/{\eps}.$ Applying Lemma \ref{pert-almost-ortho} to $N$ and $\eps/16$, one obtains $\delta_1$.
Choose $\delta>0$ such that
$$16\delta< \delta_1 \quad \mathrm{and} \quad \frac{\eps}{8} + \delta + 4((4\delta+\frac{\eps}{8}) + \frac{\eps}{16}) < \eps. $$
 Then the pair $(N, \delta)$ has the property of the lemma.

Since $a_i a_j = a_j$, $i < j$, by \eqref{here-cond-1-a},
$$a_i d_j \approx_\delta a_i a_{j-1} d_j = a_{j-1} d_j \approx_\delta d_j,\quad i< j. $$ Hence, together with \eqref{here-cond-1-a},
\begin{equation}\label{here-cond-1-a-eq}
\norm{a_id_j - d_j} < 2\delta,\quad  i < j.
\end{equation}
A similar argument applied to \eqref{here-cond-1-b} shows
\begin{equation}\label{here-cond-1-b-eq}
\norm{d_ia_j - a_j} < 2\delta, \quad  i < j.
\end{equation}

Also note that, if $i < j-1$, then
$$d_i d_j \approx_\delta d_i a_{j-1} d_j \approx_{2\delta} a_{j-1} d_j \approx_\delta d_j,$$
and therefore, a straightforward calculation shows that 
\begin{equation}\label{orth-cond}
\norm{(d_i-d_{i+1})(d_j - d_{j+1})} <16\delta < \delta_1 ,\quad i < j-2.
\end{equation}

Note that $$a \approx_{\frac{2}{N}} aa_2,$$
and then, together with \eqref{here-cond-1-b}, one has
\begin{equation}\label{first-bump}
a \approx _{\frac{2}{N}} aa_2 \approx_{\delta } a a_2d_1 \approx_{\frac{2}{N}} ad_1.
\end{equation}
Noting that $$a = \frac{1}{N}(a_1 + \cdots + a_N),$$ together with \eqref{here-cond-1-a-eq} and  \eqref{here-cond-1-b-eq}, 
one also has that, for $i=1, ..., N-1$, 
\begin{eqnarray*}
& & a(d_i - d_{i+1})  \\
& = & \frac{1}{N}(a_1 + \cdots + a_N)(d_i - d_{i+1}) \nonumber \\
& = & \frac{1}{N}((a_1 d_i + \cdots + a_Nd_i) - (a_1 d_{i+1} + \cdots + a_Nd_{i+1})) \nonumber \\
& \approx_{4\delta} & \frac{1}{N}((\underbrace{d_i + \cdots + d_{i}}_{i-1} + a_id_i + a_{i+1} + \cdots + a_N) \\
&&  - (\underbrace{d_{i+1} + \cdots + d_{i+1}}_{i} +  a_{i+1}d_{i+1} + a_{i+2}+ \cdots + a_N)) \nonumber \\
& = & \frac{i}{N}(d_i - d_{i+1}) + \frac{1}{N}(a_{i+1}-d_{i} + a_id_i - a_{i+1}d_{i+1}) \nonumber \\
&\approx_{\frac{4}{N}}&  \frac{i}{N}(d_i - d_{i+1}). \nonumber
\end{eqnarray*}
That is, 
\begin{equation}\label{local-bump}
\norm{a(d_i - d_{i+1}) -  \frac{i}{N}(d_i - d_{i+1})} < 4\delta + \frac{4}{N} < 4\delta + \frac{\eps}{8},\quad i=1, 2, ..., N-1.
\end{equation}
A similar argument also shows
\begin{equation}\label{local-bump-1}
\norm{ad_N - d_N} < 2\delta + \frac{2}{N} < 4\delta + \frac{\eps}{8}.
\end{equation}

Then, on applying \eqref{first-bump}, \eqref{local-bump}, \eqref{local-bump-1}, \eqref{orth-cond}, and Lemma \ref{pert-almost-ortho} (note that $a, d_1, ..., d_N$ commute),
\begin{eqnarray*}
a & \approx_{\frac{4}{N} + \delta} & ad_1\\
& = & a(d_1 - d_2 + d_2 - d_3 + \cdots + d_{N-1} - d_{N} + d_{N}) \\
& = & a(d_1 - d_2) + a(d_2 - d_3) + \cdots + a(d_{N-1} - d_{N}) + ad_{N} \\
& \approx_{4((4\delta + \frac{\eps}{8}) + \frac{\eps}{16})} & \frac{1}{N}(d_1 - d_2) + \frac{2}{N}(d_2 - d_3)  + \cdots + \frac{N-1}{N}(d_{N-1} - d_{N}) + d_{N} \\
& = & \frac{1}{N}(d_1 + d_2 + \cdots + d_{N}),
\end{eqnarray*}
as desired.
\end{proof}

\begin{lem}\label{hereditary-sets-tr-app}
Let $(D, A)$ be a pair of unital C*-algebras, where $D$ is commutative. For any $\eps>0$, there are $\delta>0$ and $N \in \mathbb N$ with the following property:  

Let $a \in D \subseteq A$ be a positive element with norm at most $1$, and set $$a_i = \chi_i(a),\quad  i=1, ..., N, $$ where $$\chi_i(t) = \left\{ 
\begin{array}{ll}
0, & t \leq \frac{i-1}{N}, \\
\mathrm{linear}, & t \in [\frac{i-1}{N}, \frac{i}{N}], \\
1, & t \geq \frac{i}{N}.
\end{array} 
\right.
$$
Let $d_1, ..., d_N \in A$ be positive elements with norm at most $1$ such that
         \begin{enumerate}
         \item $\mathrm{dist}_{2, \tr(A)}(d_i, (D)^+_1) < \delta$, $i=1, 2, ..., N$,
         \item\label{here-cond-1-a-tr} $ \norm{a_{i} d_{i+1} - d_{i+1}}_{2, \tr(A)} < \delta$, $i=1, ..., N-1$, and 
         \item\label{here-cond-1-b-tr} $ \norm{d_i a_{i+1} - a_{i+1}}_{2, \tr(A)} < \delta$, $i=1, ..., N-1$.
         \end{enumerate} 
        Then
        $$\norm{ a  - \frac{1}{N}(d_1 + \cdots + d_N)}_{2, \tr(A)} < \eps.$$
\end{lem}

\begin{proof}
Applying Lemma \ref{hereditary-sets-norm-app} to $\eps/2$, one obtains, say, the pair $(N, \delta_2)$. Applying Lemma \ref{tr-to-norm-red} to $\min\{\delta_2, \eps/4\}$, one obtains $\delta_1$. Set $\delta = \min\{\delta_1/4, \eps/4\}$. Then $(N, \delta)$ has the property of the lemma.

Indeed, let $d_1, ..., d_N$ be given. Since $\mathrm{dist}_{2, \tr(A)}(d_i, (D)^+_1) < \delta$, $i=1, 2, ..., N$, there are positive contractions $\tilde{d}_1, ... \tilde{d}_N \in D$ such that
\begin{equation}\label{pert-lem-cond-1}
\norm{d_i - \tilde{d}_i}_{2, \tr(A)} < \min\{\delta_1/4, \eps/4\},\quad i=1, ..., N.
\end{equation}
Then, it follows from \eqref{here-cond-1-a-tr} and \eqref{here-cond-1-b-tr} that
$$  \norm{a_{i} \tilde{d}_{i+1} - \tilde{d}_{i+1}}_{2, \tr(A)} < \delta_1 \quad \mathrm{and} \quad \norm{\tilde{d}_i a_{i+1} - a_{i+1}}_{2, \tr(A)} < \delta_1,\quad i=1, 2, ..., N-1. $$ 
Applying Lemma \ref{tr-to-norm-red} to each element $\tilde{d}_i$, $i=1, ..., N$, one obtains a positive contraction $\tilde{\tilde{d}}_i \in D$ such that
\begin{equation}\label{pert-lem-cond-2}  
 \norm{\tilde{d}_i - \tilde{\tilde{d}}_i}_{2, \tr(A)} < \eps/4,\quad i=1, ..., N 
 \end{equation}
and
$$  \norm{a_{i} \tilde{\tilde{d}}_{i+1} - \tilde{\tilde{d}}_{i+1}} < \delta_2 \quad \mathrm{and} \quad \norm{\tilde{\tilde{d}}_i a_{i+1} - a_{i+1}} < \delta_2,\quad i=1, 2, ..., N-1. $$ 
Then, by Lemma \ref{hereditary-sets-norm-app}, one has
$$\norm{ a  - \frac{1}{N}(\tilde{\tilde{d}}_1 + \cdots + \tilde{\tilde{d}}_N)} <  \frac{\eps}{2},$$
and hence, together with \eqref{pert-lem-cond-1} and \eqref{pert-lem-cond-2}, one has
$$\norm{ a  - \frac{1}{N}(d_1 + \cdots + d_N)}_{2, \tr(A)} < \frac{\eps}{2} + \frac{\eps}{2} = \eps,$$
as desired.
\end{proof}

\begin{lem}\label{hereditary-sets-preserve-tr}
Let $A$ be a unital C*-algebra. For any $\eps>0$, any $N \in \mathbb N$, and any $\chi\in \mathrm{C}([0, 1])^+$, there are $\delta>0$ and $M \in \mathbb N$ such that if $b_1, b_2, ..., b_N \in A$ and $d_1, d_2, ..., d_N \in A$ are positive elements with norm at most $1$ such that
\begin{enumerate}
\item $ \norm{d_id_{i+1} - d_{i+1}}_{2, \tr(A)} < \delta$, $i=1, ..., N-1$,  
\item $b_i b_{i+1} = b_{i+1}$, $i=1, 2, ..., N$, and
\item $\abs{\tau(b_i^j) - \tau(d_i^j)} < \delta$, $i=1, 2, ..., N$, $j=1, ..., M$, $\tau \in \tr(A)$,
\end{enumerate}
then
$$\abs{\tau(\chi(\frac{1}{N}(b_1 + \cdots + b_N))) - \tau(\chi(\frac{1}{N}(d_1 + \cdots + d_N))) } < \eps,\quad \tau \in \tr(A).$$
\end{lem}

\begin{proof}
It is enough to prove the statement for $\chi$ a monomial, i.e., $\chi(t) = t^n$. Note that there are positive numbers $\alpha_{i, j}$, $i=1, ..., N$, $j=1, ..., n$, such that
\begin{eqnarray*}
b^n & = & (\frac{1}{N}(b_1 + \cdots + b_N))^n \\
& = & \frac{1}{N^n}\sum_{i_1+\cdots + i_N = n} b_1^{i_1}\cdots b_N^{i_N} \\
& = & \sum_{i=1}^N \sum_{j=1}^n \alpha_{i, j} b_i^j.
\end{eqnarray*}
Hence,
$$\tau(b^n) = \sum_{i=1}^N \sum_{j=1}^n \alpha_{i, j} \tau(b_i^j).$$

Then there is $\delta>0$ such that if $$ \norm{d_id_{i+1} - d_{i+1}}_{2, \tr(A)} < \delta,\quad i=1, ..., N-1,$$ then
$$ \norm{ (\frac{1}{N}(d_1 + \cdots + d_N))^n - \sum_{i=1}^N \sum_{j=1}^n \alpha_{i, j} d_i^j }_{2, \tr(A)} < \eps/2.$$ 
In particular,
$$\tau((\frac{1}{N}(d_1+\cdots + d_N))^n) \approx_{\eps/2} \tau( \sum_{i=1}^N \sum_{j=1}^n \alpha_{i, j} d_i^j) = \sum_{i=1}^N \sum_{j=1}^n \alpha_{i, j} \tau(d_i^j). $$

Moreover, one may assume that $\delta>0$ is sufficiently small such that if 
$$\abs{\tau(b_i^j) - \tau(d_i^j)} < \delta,\quad i=1, 2, ..., N,\ j=1, ..., n,$$
then
$$ \abs{ \sum_{i=1}^N \sum_{j=1}^n \alpha_{i, j} \tau(b_i^j) - \sum_{i=1}^N \sum_{j=1}^n \alpha_{i, j} \tau(d_i^j) } < \eps/2.$$
Then this $\delta$ and $M:=n$ have the desired property.
\end{proof}

\section{The small boundary property}

%
%

Let us show that Property (S) (Definition \ref{defn-S}) for the ambient C*-algebra $A$ indeed implies the (SBP) for the subalgebra $D$ when Properties (C) and (E) (Definitions \ref{defn-C} and \ref{defn-D} ) are present (Theorem \ref{prop-S}). Since the crossed product C*-algebra $\mathrm{C}(X) \rtimes \Gamma$ always has Property (E) and Property (S) (Theorem \ref{S-holds}) if $(X, \Gamma)$ has the (URP), together with Theorem \ref{SBP=>C}, this implies that Property (C) is equivalent to the (SBP) in this setting.

For each $\eps>0$, define 
\begin{equation}\label{defn-eta}
\eta_\eps(t) = \left\{
\begin{array}{ll}
0, & t \leq 1-\eps, \\
\mathrm{linear}, & t \in [1-\eps, 1-\eps/2], \\
1, & t \geq 1-\eps/2.
\end{array}
\right.
\end{equation}

In the proof of the following lemma, we use $O(\eps)$ to denote a quantity which converges to $0$ when $\eps$ approaches $0$.
\begin{lem}\label{D-sandwich}
Let $A$ be a unital C*-algebra, and let $D \subseteq A$ be a unital commutative subalgebra such that the pair $(D, A)$ has Property (C). Let $\phi_1, \phi_2, \phi_3, \psi_1, \psi_2 \in (D)_1^+$ and $\eps_0>0$ have the following properties:
\begin{enumerate}
\item\label{cond-s-1} $\phi_2, \phi_3 \in \overline{\phi_1 D\phi_1}$,
\item\label{cond-s-2} $\psi_1\psi_2 = \psi_2$, $\phi_2 \phi_3 = \phi_3$,
\item\label{cond-s-3} $\mathrm{d}_\tau(\psi_1) < \mathrm{d}_\tau(\phi_1)$ for all $\tau\in\tr(A)$,
\item\label{cond-s-4} $\inf\{\tau(\psi_2) - \mathrm{d}_\tau(\phi_3): \tau \in \tr(A)\} >0 $, and
\item\label{cond-s-5} $\inf\{\mathrm{d}_\tau(\phi_2) - \tau(\eta_{\eps_0}(\psi_1)): \tau \in \tr(A)\} >0 $.  \end{enumerate}
Then, for any $\eps>0$, there is a contraction $u\in A$ such that
$$ u^*\psi_1u \in_\eps^{\norm{\cdot}_2}  \overline{\phi_1A\phi_1},\quad u^*\eta_{\eps_0/2}(\psi_1)u \in^{\norm{\cdot}_2}_{\eps} \overline{\phi_2A\phi_2}, \quad  \eta_{\eps}(u^*\psi_1u) \phi_3 \approx_{\eps}^{\norm{\cdot}_2} \phi_3,$$
and
$$\mathrm{dist}_{2, \tr(A)}(udu^*, (D)_1) < \eps,\quad \mathrm{dist}_{2, \tr(A)}(u^*du, (D)_1) < \eps,\quad d \in (D)_1,$$
and
$$\norm{uu^* - 1}_{2, \tr(A)}, \quad \norm{u^*u - 1}_{2, \tr(A)} < \eps.$$
\end{lem}

\begin{proof}
Let $\eps>0$ be given. Choose $\delta \in (0, \min\{\eps_0, \eps\}/4)$ such that if $x, y$ are positive contractions of a unital C*-algebra $A$ such that $\norm{x - y}_{2, \tr(A)} < \delta$, then
\begin{equation}\label{pre-pert-delta}
\norm{\eta_\eps(x) - \eta_\eps(y)}_{2, \tr(A)} < \eps/2.
\end{equation}

Define
$$ \delta_1 := \inf\{\tau(\psi_2) - \mathrm{d}_\tau(\phi_3): \tau \in \tr(A)\} >0$$
and
$$\delta_2 := \inf\{\mathrm{d}_\tau(\phi_2) - \tau(\eta_{\eps_0}(\psi_1)): \tau \in \tr(A)\} > 0. $$

By Property (C) and Assumption (\ref{cond-s-3}), for any $\eps'>0$ (to be determined later), there is a contraction $u_1\in A$ such that
\begin{equation}\label{approx-C-1-h}
u^*_1 \psi_1 u_1 \in^{\norm{\cdot}_2}_{\eps'} \overline{\phi_1 A \phi_1}, 
\end{equation}
\begin{equation}\label{approx-C-1-n}
\mathrm{dist}_{2, \tr(A)}(u_1du^*_1, (D)_1) < \eps', \quad \mathrm{dist}_{2, \tr(A)}(u_1^*du_1, (D)_1) < \eps',\quad d \in (D)_1, 
\end{equation}
and
\begin{equation}\label{approx-C-1-u}
\norm{u_1u_1^* - 1}_{2, \tr(A)},\quad  \norm{u_1^*u_1 - 1}_{2, \tr(A)} < \eps'.
\end{equation}

By \eqref{approx-C-1-h}, there is $n$ large enough that $$ \norm{(\phi_1^{\frac{1}{n}}) (u_1^*\psi_1u_1) - u_1^*\psi_1u_1}_{2, \tr(A)} < \eps'. 
$$
By \eqref{approx-C-1-n}, there is a positive contraction $d_1 \in D$ such that
\begin{equation*}
\norm{d_1- u_1^*\psi_1u_1}_{2, \tr(A)}< \eps', 
\end{equation*}
and then
$$\norm{\phi_1^{\frac{1}{n}}d_1 - u_1^*\psi_1u_1}_{2, \tr(A)} < 2\eps'. 
$$
Consider $\phi_1^{\frac{1}{n}}d_1$, and still denote this element by $d_1$. One has $$d_1 \in \overline{\phi_1 D \phi_1}$$
and
\begin{equation}\label{approx-C-1}
\norm{d_1 - u_1^*\psi_1u_1}_{2, \tr(A)} < 2\eps'. 
\end{equation}

Consider  the contraction
$$\eta_{\eps_0}(u^*_1\psi_1 u_1),$$
and note that (by \eqref{approx-C-1}, \eqref{approx-C-1-u} and Condition \eqref{cond-s-5})
$$ \mathrm{d}_\tau( \eta_{\eps_0/2}(d_1) ) \leq \tau(\eta_{\eps_0}(d_1))  \approx_{O(\eps')}^{\norm{\cdot}_2}  \tau( \eta_{\eps_0}(u^*_1\psi_1 u_1) ) \approx_{O(\eps')}^{\norm{\cdot}_2} \tau( \eta_{\eps_0}(\psi_1) ) < \mathrm{d}_\tau(\phi_2) - \delta_2/2,\quad \tau \in \tr(A).$$
With $\eps'$ sufficiently small, one has 
$$\mathrm{d}_\tau(\eta_{\eps_0/2}(d_1)) < \mathrm{d}_\tau(\phi_2),\quad \tau \in \tr(A).$$ 
By Property (C), for any $\eps''>0$ (to be determined later), there is $u_2 \in \overline{\phi_1 A \phi_1}+\Comp 1$ such that 
\begin{equation}\label{approx-C-2-here}
u^*_2 \eta_{\eps_0/2}(d_1) u_2 \in^{\norm{\cdot}_2}_{\eps''} \overline{\phi_2 A \phi_2},
\end{equation}
\begin{equation}\label{approx-C-2-n}
\mathrm{dist}_{2, \tr(A)}(u_2 d u^*_2, D_1) < \eps'', \quad \mathrm{dist}_{2, \tr(A)}(u_2^*du_2, D_1) < \eps'',\quad d \in D_1,
\end{equation}
and
\begin{equation}\label{approx-C-2-u}
\norm{u_2u_2^* - 1}_{2, \tr(A)},\quad  \norm{u_2^*u_2 - 1}_{2, \tr(A)} < \eps''.
\end{equation}

Since (as $\delta < \eps_0/4$)
$$\eta_{\eps_0/2}(d_1) \eta_\delta(d_1) = \eta_\delta(d_1),$$
it follows from \eqref{approx-C-2-here} that
\begin{equation}\label{close-hereditary}
u_2^* \eta_\delta(d_1) u_2 \in^{\norm{\cdot}_2}_{2\eps''} \overline{\phi_2 A \phi_2}. 
\end{equation}

By \eqref{approx-C-2-n}, there are positive contractions $$ d_{1, 1},\  d_{1, \delta} \in  D $$ such that
\begin{equation}\label{pert-C-step-2}
 \norm{d_{1, 1} - u_2^* d_1 u_2}_{2, \tr(A)} < \eps'',\quad  \norm{d_{1, \delta} - u_2^*\eta_{\delta}(d_1)u_2}_{2, \tr(A)} < \eps''. 
\end{equation}  
 By \eqref{close-hereditary}, $$d_{1, \delta} \in_{3\eps''}^{\norm{\cdot}_2} \in \overline{\phi_2A\phi_2}.$$
 With the same argument as above for $d_1$, there is a positive contraction, still denoted by $ d_{1, \delta}$, such that 
 \begin{equation}\label{u-2-2-approx-C-1}
 d_{1, \delta} \in \overline{\phi_2 D \phi_2}\quad \mathrm{and}\quad \norm{d_{1, \delta} - u_2^*\eta_{\delta}(d_1)u_2}_{2, \tr(A)} < 5\eps''.
 \end{equation}
Also note that, by \eqref{approx-C-2-u}, $$ \norm{ \eta_\delta(d_{1, 1}) - d_{1, \delta} }_{2, \tr(A)} = O(\eps''). $$

Define
$$\tilde{d}_{1, 1} = f_{\delta}(d_{1, 1}) \quad \mathrm{and} \quad \tilde{d}_{1, \delta} = \eta_{\delta}(d_{1, 1}),$$
where
$$
f_{\delta}(t) = \left\{ 
\begin{array}{ll}
0, & t \leq 0, \\
\mathrm{linear}, & t\in [0, 1-\delta], \\
1, & t \geq 1-\delta.
\end{array}
\right.
$$ 
Then
\begin{equation}\label{pert-fur-hereditary}
 \tilde{d}_{1, 1} \tilde{d}_{1, \delta} = \tilde{d}_{1, \delta} \quad \mathrm{and} \quad \norm{\tilde{d}_{1, 1} - d_{1, 1}} < \delta,
 \end{equation}
and by \eqref{pert-C-step-2},
$$ \norm{\tilde{d}_{1, \delta} - d_{1, \delta}}_{2, \tr(A)} \approx_{O(\eps'')} \norm{\eta_\delta(u_2^* d_1 u_2) -  u^*_2\eta_\delta(d_1)u_2} = O(\eps'').$$
Then, with $\eps''$ sufficiently small, there is $n \in \mathbb N$ such that $$ \norm{(\tilde{d}_{1, \delta})^{\frac{1}{n}} d_{1, \delta} - d_{1, \delta}}_{2, \tr(A)} < \delta_1/4, $$
and hence, for all $\tau \in \tr(A)$,
\begin{eqnarray*}
\mathrm{d}_\tau((\tilde{d}_{1, \delta})^{\frac{1}{n}} d_{1, \delta}) + \delta_1/4 & \geq & \tau((\tilde{d}_{1, \delta})^{\frac{1}{n}} d_{1, \delta}) + \delta_1/4   \\
& > & \tau(d_{1, \delta}) \approx_{5\eps''} \tau(u_2^* \eta_{\delta}(d_1)u_2) \quad \quad (\eqref{u-2-2-approx-C-1}) \\
& \approx_{O(\eps')} &  \tau(u^*_2\eta_{\delta}(u_1^* \psi_1 u_1 )u_2) \quad\quad (\eqref{approx-C-1}) \\
& \approx_{O(\eps' + \eps'')} & \tau(\eta_{\delta}(\psi_1 ))  > \tau(\psi_2) \\ 
& > & \mathrm{d}_\tau(\phi_3) + \delta_1/2.
\end{eqnarray*}
With $\eps'$ and $\eps''$ sufficiently small, one has 
$$ \mathrm{d}_\tau((\tilde{d}_{1, \delta})^{\frac{1}{n}} d_{1, \delta}) > \mathrm{d}_\tau(\phi_3),\quad \tau \in \tr(A). $$

Note that $(\tilde{d}_{1, \delta})^{\frac{1}{n}} d_{1, \delta} \in \overline{\phi_2D\phi_2}$ (both $(\tilde{d}_{1, \delta})^{\frac{1}{n}}$ and  $d_{1, \delta}$ belong to $D$, so they commute). By Property (C), for any $\eps'''>0$ (to be determined later), there is $u_3 \in \overline{\phi_2 A \phi_2} + \Comp 1$ such that 
\begin{equation}\label{approx-C-3-h}
u_3\phi_3u_3^* \in _{\eps'''}^{\norm{\cdot}_2} \overline{(\tilde{d}_{1, \delta})^{\frac{1}{n}} d_{1, \delta} D (\tilde{d}_{1, \delta})^{\frac{1}{n}} d_{1, \delta}},
\end{equation}
$$\mathrm{dist}_{2, \tr(A)}(u_3 d u^*_3, D_1) < \eps''', \quad \mathrm{dist}_{2, \tr(A)}(u_3^*du_3, D_1) < \eps''',\quad d \in D_1, $$
and
\begin{equation}\label{approx-C-3-u}
\norm{u_3u_3^* - 1}_{2, \tr(A)},\quad  \norm{u_3^*u_3 - 1}_{2, \tr(A)} < \eps'''.
\end{equation}

Note that, by \eqref{pert-fur-hereditary},
$$\tilde{d}_{1, 1}c = c,\quad c \in \overline{(\tilde{d}_{1, \delta})^{\frac{1}{n}} d_{1, \delta} A (\tilde{d}_{1, \delta})^{\frac{1}{n}} d_{1, \delta} }.$$
By \eqref{approx-C-3-h}, there is $c \in \overline{(\tilde{d}_{1, \delta})^{\frac{1}{n}} d_{1, \delta}D (\tilde{d}_{1, \delta})^{\frac{1}{n}} d_{1, \delta}}$ such that $\norm{c} \leq 1$ and
$$ \norm{c - u_3\phi_3 u_3^*}_{2, \tr(A)} < \eps'''.$$ 
Then 
\begin{eqnarray}\label{comp-approx}
 ( u^*_2\eta_{\eps}(u_1^* \psi_1 u_1)u_2) (u_3 \phi_3 u_3^*) & \approx_{O(\eps')}^{\norm{\cdot}_2} &  ( u^*_2\eta_{\eps}(d_1)u_2) (u_3 \phi_3 u_3^*) \quad\quad ( \eqref{approx-C-1} )   \\
 & \approx_{O(\eps'')}^{\norm{\cdot}_2} & \eta_{\eps}(u_2^*d_1 u_2) (u_3 \phi_3 u_3^*) \quad\quad (\eqref{approx-C-2-u})   \nonumber \\
 & \approx_{O(\eps'')}^{\norm{\cdot}_2} & \eta_{\eps}(d_{1, 1}) (u_3 \phi_3 u_3^*) \quad\quad (\eqref{pert-C-step-2})   \nonumber \\
 & \approx_{\eps/2+O(\eps'')}^{\norm{\cdot}_2} &  \eta_{\eps}(\tilde{d}_{1, 1}) (u_3 \phi_3 u_3^*) \quad \quad (\eqref{pert-fur-hereditary} \eqref{pre-pert-delta}) \nonumber \\
 & \approx_{\eps'''}^{\norm{\cdot}_2} &  \eta_{\eps}(\tilde{d}_{1, 1})  c \nonumber \\
 &=&c \approx_{\eps'''}^{\norm{\cdot}_2} u_3 \phi_3 u_3^*. \nonumber
\end{eqnarray}

Then, consider the contraction $$ u = u_1u_2u_3.$$
Since 
$$u_2 \in \overline{\phi_1 A \phi_1}+\Comp 1\quad \mathrm{and}\quad u_3 \in \overline{\phi_2 A \phi_2} + \Comp 1, $$ one has
$$(u_1u_2u_3)^*\psi_1(u_1u_2u_3) \in \overline{\phi_1 A \phi_1}.$$
Moreover,
\begin{eqnarray*}
(u_1u_2u_3)^*\eta_{\eps_0/2}(\psi_1)(u_1u_2u_3)  & \approx_{O(\eps')}^{\norm{\cdot}_2} & u^*_3u_2^*\eta_{\eps_0/2}(u_1^*\psi_1 u_1)u_2 u_3 \quad\quad (\eqref{approx-C-1-u}) \\
& \approx_{O(\eps')}^{\norm{\cdot}_2} & u^*_3(u_2^*\eta_{\eps_0/2}(d_1)u_2) u_3 \quad\quad (\eqref{approx-C-1}) \\
& \in_{\eps''}^{\norm{\cdot}_2} & \overline{\phi_2 A \phi_2}, \quad\quad (\eqref{approx-C-2-here})
\end{eqnarray*}
and 
\begin{eqnarray*}
\eta_{\eps}((u_1u_2u_3)^*\psi_1(u_1u_2u_3)) \phi_3 & \approx_{O(\eps' + \eps'')}^{\norm{\cdot}_2} & u_3^*u_2^*\eta_\eps(u_1^*\psi_1 u_1)u_2u_3 \phi_3 \quad\quad (\eqref{approx-C-1-u}, \eqref{approx-C-2-u}) \\
 & \approx_{\eps'''}^{\norm{\cdot}_2} & u^*_3 (u^*_2  \eta_{\eps}(u_1^* \psi_1 u_1)u_2)(u_3 \phi_3 u_3^*) u_3 \quad\quad (\eqref{approx-C-3-u}) \\
 & \approx_{\eps + O(\eps'+\eps''+\eps''')} & u_3^* (u_3\phi_3 u_3^*) u_3= \phi_3. \quad\quad(\eqref{comp-approx})
 \end{eqnarray*}
With $\eps', \eps'', \eps'''$ sufficiently small, the contraction $u$ has the desired property.
\end{proof}

For technical reasons, we also need the following lemma which, very roughly, asserts that, after a perturbation with respect to the uniform trace norm, the spectrum of a positive element of the subalgebra $D$ is, in a strong sense, dense.
\begin{lem}\label{dense-sp-pert}
Let $A$ be a unital C*-algebra and let $D = \mathrm{C}(X) \subseteq A$ be a unital commutative sub-C*-algebra. Assume $A$ is simple, $X$ has no isolated points, and the pair $(D, A)$ has Property (E). Then, for any positive contraction $g \subseteq D$, any finite set $\{x_1, ..., x_n\} \subseteq (0, 1]$, and any $\eps>0$, there is a positive contraction $\tilde{g} \in D$ such that 
\begin{enumerate}
\item $\norm{\tilde{g} - g}_{2, \tr(A)} < \eps,$ and 
\item each point $x_i$, $i=1, ..., n$, is in $\mathrm{sp}(\tilde{g})$, and is not isolated from the left inside $\mathrm{sp}(\tilde{g})$ (i.e., $(s, x_i) \cap \mathrm{sp}(\tilde{g}) \neq \O$ for all $s < x_i$).
\end{enumerate}
\end{lem}
\begin{proof}
Since $A$ is simple (and non-elementary), there are mutually orthogonal positive elements $a_1, ..., a_n \in A$ such that $$\norm{a_i} = 1\quad  \mathrm{and} \quad \mathrm{d}_\tau(a_i) < \eps/n^2,\quad i=1, ..., n,\ \tau \in\tr(A). $$ (See, for instance, Lemma 4.7 of \cite{Niu-MD-Z}.)

Consider the contraction $$a= \frac{1}{n} a_1 + \cdots + \frac{n}{n} a_n,$$ and note that it has the property 
$$\quad 0 < \tau(h_i(a)) < \eps/n,\quad i=1, ..., n,\ \tau \in\tr(A),$$ where 
$h_i: [0, 1] \to [0, 1]$ is the continuous function taking value $1$ at $[(4i-1)/4n, (4i+1)/4n]$, $0$ on $[0, (2i-1)/2n]$ and $[(2i+1)/2n, 1]$, and linear between. By Property (E), there is a positive contraction $d \in D$ such that
$$\quad 0 < \tau(h_i(d)) < \eps/n,\quad i=1, ..., n,\ \tau \in\tr(A),$$
and there are mutually orthogonal positive elements $b_1, ..., b_n \in D$ such that
$$ \norm{b_i} = 1 \quad \mathrm{and} \quad \mathrm{d}_\tau(b_i)<\eps/n,\quad i=1, ..., n,\ \tau \in\tr(A). $$

Consider the sets  $$U_i = b^{-1}_i((0, 1]) \quad \mathrm{and} \quad V_i = b_i^{-1}((1/2, 1]),\quad i=1, ..., n. $$ Then
$$\overline{V_i} \subseteq U_i \quad\mathrm{and} \quad \mu_\tau(U_i) < \eps/n,\quad i=1, ..., n,\ \tau \in\tr(A).$$
For each $V_i$, $i=1, ..., n$, since $X$ has no isolated points, there is a continuous function $g_i: X \to [0, 1]$ such that $ g_i|_{V_i^c} = 0$ and $1$ is not isolated from the left in $g_i(X)$ (i.e., $(s, 1) \cap g_i(X) \neq \O$ for all $s<1$). Also pick a continuous function $r: X \to [0, 1]$ such that $$r|_{X \setminus (\bigcup_{i=1}^n U_i)} = 1\quad \mathrm{and}\quad r|_{\bigcup_{i=1}^n \overline{V_i}} = 0.$$
Then the function 
$$ \tilde{g}: = g r + (x_1g_1 + x_2g_2 + \cdots + x_ng_n)$$
has the desired property.
\end{proof}

We are now ready to prove the main theorem of the paper, which states that Property (S) of $A$ implies the (SBP) of $(D, \tr(A))$ if Properties (C) and (E) are present.

\begin{thm}\label{prop-S}
Let $A$ be a unital simple C*-algebra, and let $D = \mathrm{C}(X) \subseteq A$ be a unital commutative sub-C*-algebra such that $X$ has no isolated points. Assume that the pair $(D, A)$ has Properties (C) and (E). Then, if the C*-algebra $A$ has Property (S), the pair $(D, \tr(A))$ has the (SBP).
\end{thm}

\begin{proof}
By Theorem \ref{SBP-2-norm}, it is enough to show that for any self-adjoint contraction $f \in D$ and any $\eps>0$,  there is a self-adjoint element $g \in D$ such that
\begin{enumerate}
\item $\norm{f - g }_{2, \tr(A)} < \eps$, and
\item there is $\delta>0$ such that $\tau(\chi_\delta(g)) < \eps$, $\tau \in \tr(A)$.
\end{enumerate}

To show this statement, it is enough to prove it for $f$ such that $\mathrm{sp}(f) = [-1, 1]$. Indeed, set 
$$t_- = \sup\{t < 0: t \notin \mathrm{sp}(h)\} \quad \mathrm{and} \quad  t_+ = \inf\{t> 0: t \notin \mathrm{sp}(h)\}.$$

If $t_- = 0$ or $t_+ = 0$, then it is straightforward to perturb $f$ to produce $g$ (with $\chi_\delta(g) = 0$). Assume neither of $t_-$ and $t_+$ is zero. Choose $s_-,\  s_+ \notin \mathrm{sp}(f)$ such that $$0\leq t_- - s_- < \min\{\eps, -t_-\} \quad \mathrm{and} \quad 0\leq s_+ - t_+ < \min\{\eps, t_+\} $$ and consider the self-adjoint element $h(f)$ where
$$ h = \left\{
\begin{array}{ll}
0, & t < t_-, \\
t_- - s_-, & t \in [t_-, s_-], \\
t, & t \in [t_-, t+-], \\
s_+ - s_-, & t \in [t_+, s_+],\\
0 & t > s_+.
 \end{array}
\right. $$
Then $\mathrm{sp}(h(f)) = [t_-, t_+]$, and $$\norm{f - (f^-_{s_-} + h(f) +f^+_{s_+} )} < \eps,$$
where $f^-_{s_-} (t)= t $ if $t < s_-$ and $f^-_{s_-} (t)= 0$ otherwise, and $ f^+_{s_+} $ is defined similarly. Then, applying the statement to the self-adjoint element $h(f)$, one obtains the desired approximation $g$. 

Now, let us assume that $\mathrm{sp}(f) = [-1, 1]$. Identifying $[-1, 1]$ with $[0, 1]$, let us show the following (equivalent) statement: 

Let $f \in D$ be a positive contraction with $\mathrm{sp}(f) = [0, 1]$, and let $\eps>0$. Then there is a positive contraction $g \in D$ such that
\begin{equation}\label{positive-s1}
\norm{f - g }_{2, \tr(A)} < \eps,
\end{equation} 
and there is $\delta>0$ such that
\begin{equation}\label{positive-s2}
\tau(\chi_{\frac{1}{2}, \delta}(g)) < \eps,\quad \tau \in \tr(A),
\end{equation}
where 
$$ \chi_{\frac{1}{2}, \delta}(t) = \left\{
\begin{array}{ll}
0, & t < \frac{1}{2} - \delta, \\
\mathrm{linear}, & t \in [\frac{1}{2} - \delta, \frac{1}{2} - \frac{\delta}{2}],\\
1, & t \in [\frac{1}{2} - \frac{\delta}{2}, \frac{1}{2} + \frac{\delta}{2}],\\
\mathrm{linear}, & t \in [\frac{1}{2} + \frac{\delta}{2}, \frac{1}{2} + {\delta}],\\
0, & t > \frac{1}{2} + \delta.
\end{array}
\right. $$

Let $(f, \eps)$ be given. Applying Lemma \ref{hereditary-sets-tr-app} to $\eps/2$ (in place of $\eps$), we obtain $N$ and  $\eps_0$ (in place of $\delta$). Choose $\eps_1>0$ such that $$3\eps_1< 1/N,\quad 8N\eps_1 < \eps_0, \quad \mathrm{and} \quad  \eps_1 < \eps/4.$$

For each $i=1, 2, ..., N,$ consider the following functions
$$\chi_i(t) = \left\{ 
\begin{array}{ll}
0, & t \leq \frac{i-1}{N}, \\
\mathrm{linear}, & t \in [\frac{i-1}{N}, \frac{i}{N}], \\
1, & t \geq \frac{i}{N},
\end{array} 
\right.
\quad
\chi_{i, \eps_1}(t) = \left\{ 
\begin{array}{ll}
0, & t \leq \frac{i-1}{N} + \eps_1, \\
\mathrm{linear}, & t \in [\frac{i-1}{N} + \eps_1, \frac{i}{N}+\eps_1], \\
1, & t \geq \frac{i}{N}+\eps_1,
\end{array} 
\right.
$$

$$\kappa_{i, \eps_1}(t) = 
\left\{ 
\begin{array}{ll}
0, & t \leq \frac{i-1}{N} + \frac{\eps_1}{2}, \\
\mathrm{linear}, & t \in [\frac{i-1}{N} +\frac{\eps_1}{2}, \frac{i-1}{N} +\eps_1], \\
1, & t \geq \frac{i-1}{N} + \eps_1,
\end{array}
\right.
\quad
\theta_{i, \eps_1} = 
\left\{
\begin{array}{ll}
0, & t \leq \frac{i}{N} + \eps_1, \\
\textrm{linear}, & t \in [\frac{i}{N} + \eps_1, \frac{i}{N} + 2\eps_1], \\
1, & t \geq \frac{i}{N} + 2 \eps_1,
\end{array}
\right.
$$

$$
\xi_{i, \eps_1}^{+} = 
\left\{
\begin{array}{ll}
0, & t \leq \frac{i}{N}, \\
\textrm{linear}, & t \in [\frac{i}{N}, \frac{i}{N} + \frac{\eps_1}{2}], \\
1, & t \geq \frac{i}{N} + \frac{\eps_1}{2},
\end{array}
\right.
\quad 
\mathrm{and}
\quad
\xi_{i, \eps_1}^{-} = 
\left\{
\begin{array}{ll}
0, & t \leq \frac{i}{N} + 2\eps_1, \\
\textrm{linear}, & t \in [\frac{i}{N} + 2\eps_1, \frac{i}{N} + 3\eps_1], \\
1, & t \geq \frac{i}{N} + 3\eps_1.
\end{array}
\right.
$$

Consider the finite set of functions 
\begin{equation}
\mathcal H = \{\chi_i,\  \chi_{i, \eps_1},\ \kappa_{i, \eps_1},\ \theta_{i, \eps_1},\ \eta_{\eps_1/2}\circ \chi_{i, \eps_1}:\ i=1, 2, ..., N\}.
\end{equation}
Note that $$\eta_{\eps_1/2}(\chi_{i, \eps_1}(t)) = 0,\quad t \leq i/N+\eps_1/2,\ i=1, ..., N,$$
where, recall (\eqref{defn-eta}),
$$
\eta_\eps(t) = \left\{
\begin{array}{ll}
0, & t \leq 1-\eps, \\
\mathrm{linear}, & t \in [1-\eps, 1-\eps/2], \\
1, & t \geq 1-\eps/2.
\end{array}
\right.
$$

Since $A$ is simple and $\mathrm{sp}(f) = [0, 1]$, there is $\gamma > 0$ such that
\begin{eqnarray}\label{small-gap}
\gamma & < & \frac{1}{4}\min\{d_\tau(\chi_i(f)) - \tau(\kappa_{i, \eps_1}(f)), \ \tau(\theta_{i, \eps_1}(f)) - \mathrm{d}_\tau(\xi_{i, \eps_1}^-(f)),\\
& & \quad \quad \quad  \tau(\xi^+_{i, \eps_1}(f)) - \tau(\eta_{\eps_1/2}(\chi_{i, \eps_1}(f))): i=1, ..., N,\ \tau\in\tr(A)\}. \nonumber 
\end{eqnarray}
Without loss of generality, one may assume that $\gamma < \eps/4$.

Since $A$ has Property (S), there is a positive contraction $\tilde{g} \in A$ such that $\norm{f - \tilde{g}}_{2, \tr(A)}$ is sufficiently small
that
\begin{equation}\label{1-pert}
\abs{ \tau(\chi(f) - \tau(\chi(\tilde{g}))} < \gamma, \quad \chi \in \mathcal H \subseteq \mathrm{C}([0, 1]),\ \tau \in \tr(A),
\end{equation}
and there is $\delta>0$ such that
\begin{equation}\label{1-small-bd}
\tau(\chi_{\frac{1}{2}, \delta}(\tilde{g})) < \eps/4,\quad \tau \in \tr(A).
\end{equation}
(Note that the choice of $\delta$ depends on $\tilde{g}$.)

Applying Lemma \ref{hereditary-sets-preserve-tr} to $\eps/4$, $N$, and $\chi_{\frac{1}{2}+\eps_1, \delta}$ (in place of $\eps$, $N$, and $\chi$, respectively), one obtains $\delta_0$ (in place of $\delta$) and $M$ which have the property specified in Lemma \ref{hereditary-sets-preserve-tr} with respect to $\eps/2$, $N$, and $\chi_{\frac{1}{2}+\eps_1, \delta}$. Choose $\delta_1>0$ such that
$$4\delta_1 < \eps_0,\quad 6\delta_1 < \delta_0, \quad \mathrm{and} \quad 3^N\delta_1 < \min\{\eps_1, \gamma/2\} . $$

Also choose $\delta_2>0$ such that if $a, b$ are positive contractions of a C*-algebra $A$, then
\begin{equation}\label{choice-delta-2}
 \norm{a - b}_{2, \tr(A)} < \delta_2 \quad \Longrightarrow \quad \norm{\chi_{\frac{1}{2} + \eps_1, \delta}(a)) - \chi_{\frac{1}{2} + \eps_1, \delta}(b)}_{2, \tr(A)} < \eps/4.
\end{equation}

Since $(D, A)$ has Property (E), there is a positive contraction $\tilde{\tilde{g}} \in D $ such that
\begin{equation}\label{2-pert}
\abs{\tau(\chi(\tilde{g})) - \tau(\chi(\tilde{\tilde{g}}))} < \gamma, \quad  \chi \in \mathcal H \cup \{ \chi_{\frac{1}{2}, \delta}\},\ \tau \in \tr(A).
\end{equation}
In particular, together with \eqref{1-pert} and \eqref{1-small-bd}, one has 
\begin{equation}\label{2-small-bd}
\abs{ \tau(\chi(f)) - \tau(\chi(\tilde{\tilde{g}})) } < 2 \gamma,\quad \chi \in \mathcal H,\ \tau \in \tr(A),
\end{equation}
and
\begin{equation}
\tau(\chi_{\frac{1}{2}, \delta}(\tilde{\tilde{g}})) < \eps/4+ \gamma < \eps/2.
\end{equation}
So
\begin{equation}\label{small-strech}
 \tau(\chi_{\frac{1}{2}+\eps_1, \delta}((\tilde{\tilde{g}} - \eps_1)_+))  < \eps/2. 
 \end{equation}

By Lemma \ref{dense-sp-pert}, after a small perturbation (with respect to $\norm{\cdot}_{2, \tr(A)}$), without loss of generality, one may assume that the numbers 
$$i/N + \eps_1,\quad i=1, 2, ..., N-1,$$
are in $\mathrm{sp}(\tilde{\tilde{g}})$, and are  not isolated from the left.

Now, let us consider the elements
$$\chi_1(f),\ \chi_2(f),\ ...,\ \chi_N \in D$$
and
$$\chi_{1, \eps_1}(\tilde{\tilde{g}}),\ \chi_{1, \eps_1}(\tilde{\tilde{g}}),\ ...,\ \chi_{1, \eps_1}(\tilde{\tilde{g}}) \in D.$$

By \eqref{2-pert}, one has
\begin{equation}
\mathrm{d}_\tau(\chi_{1, \eps_1}(\tilde{\tilde{g}})) < \tau(\kappa_{1, \eps_1}(\tilde{\tilde{g}})) \approx_{2 \gamma} \tau(\kappa_{1, \eps_1}(f)) < \mathrm{d}_\tau(\chi_1(f)),\quad \tau\in \tr(A).
\end{equation}

By the choice of $\gamma$ (\eqref{small-gap}), one has
\begin{equation}
\mathrm{d}_\tau(\chi_{1, \eps_1}(\tilde{\tilde{g}})) < \mathrm{d}_\tau(\chi_1(f)),\quad \tau\in\tr(A).
\end{equation}

Note that, by the construction of $\xi_{1, \eps_1}^+$, $\xi_{1, \eps_1}^-$, and $\theta_{1, \eps_1}$, we have 
$$   \xi_{1, \eps_1}^+(f),\ \xi_{1, \eps_1}^-(f) \in \overline{\chi_1(f) D \chi_1(f) },  $$

$$ \chi_{1, \eps_1}(\tilde{\tilde{g}}) \theta_{1, \eps_1}(\tilde{\tilde{g}}) = \theta_{1, \eps_1}(\tilde{\tilde{g}}),\quad \xi_{1, \eps_1}^+(f) \xi_{1, \eps_1}^-(f) = \xi_{1, \eps_1}^-(f),  $$

$$ \tau(\theta_{1, \eps_1}(\tilde{\tilde{g}})) \approx_{2\gamma} \tau(\theta_{1, \eps_1}(f))  >   \mathrm{d}_\tau(\xi_{1, \eps_1}^-( f )),\quad \tau \in \tr(A), $$
and
$$ \tau(\eta_{\eps_1/2}(\chi_{1, \eps_1}(\tilde{\tilde{g}}))) \approx_{2 \gamma} \tau(\eta_{\eps_1/2}(\chi_{1, \eps_1}(f))) < \tau( \xi_{1, \eps_1}^+(f)) < \mathrm{d}_\tau( \xi_{1, \eps_1}^+(f)),\quad \tau\in \tr(A). $$

By Lemma \ref{D-sandwich}, for any $\delta''>0$ (to be fixed later), there is a contraction $u_1 \in A$ such that
\begin{equation}\label{u-1-cond-1}
 u_1^*\chi_{1, \eps_1}(\tilde{\tilde{g}})u_1 \in^{\norm{\cdot}_2}_{\delta''}   \overline{\chi_1(f)A\chi_1(f)},
 \end{equation}
\begin{equation}\label{u-1-cond-2}
\eta_{\eps_1/4}(u^*_1 \chi_{1, \eps_1}(\tilde{\tilde{g}}) u_1) \in^{\norm{\cdot}_2}_{\delta''} \overline{\xi_{1, \eps_1}^+(f) A \xi_{1, \eps_1}^+(f)},
\end{equation}
\begin{equation}\label{u-1-cond-3}
 \eta_{\delta''}(u_1^*\chi_{1, \eps_1}(\tilde{\tilde{g}})u_1) \xi_{1, \eps_1}^-(f)  \approx_{\delta''}^{\norm{\cdot}_2} \xi_{1, \eps_1}^-(f),
 \end{equation}
\begin{equation}\label{u-1-cond-4}
\mathrm{dist}_{2, \tr(A)}(u_1du_1^*, (D)_1) < \delta'',\quad \mathrm{dist}_{2, \tr(A)}(u_1^*du_1, (D)_1) < \delta'',\quad d \in (D)_1,
\end{equation}
and
\begin{equation}\label{u-1-cond-5}
\norm{u_1u_1^* - 1}_{2, \tr(A)},\  \norm{u_1^*u_1 - 1}_{2, \tr(A)} < \delta''.
\end{equation}

With $\delta''$ sufficiently small, one has  
\begin{equation}\label{u-1-match-tr-p}
\abs{\tau((u_1^* x u_1)^j) - \tau(x^j)} < \delta_1,\quad j=1, ..., M,\  x\in (A)_1,\ \tau \in \tr(A),
\end{equation}
and (by \eqref{u-1-cond-4} and \eqref{u-1-cond-5})
\begin{equation}\label{u-1-close-D}
\mathrm{dist}_{2, \tr(A)}(u_1^*\chi_{1, \eps_1}(\tilde{\tilde{g}}) u_1, (D)^+_1) < 3\delta''< \min\{\eps_0, \delta_2, \eps/4\}.
\end{equation}

Set 
$$ \rho_1:=\min\{\tau(\rho_{1, \delta_1}(\tilde{\tilde{g}})): \tau \in \tr(A)\},$$
where
$$\rho_{1, \delta_1} = \left\{
\begin{array}{ll}
0, & t\leq \frac{1}{N}+\eps_1-\delta_1, \\
\mathrm{linear}, & \frac{1}{N} + \eps_1-\delta_1 \leq t \leq \frac{1}{N} + \eps_1-\delta_1/2, \\
1, & t = \frac{1}{N} + \eps_1-\delta_1/2, \\
\mathrm{linear}, & \frac{1}{N} + \eps_1-\delta_1/2 \leq t \leq \frac{1}{N} + \eps_1, \\
0, & t \geq \frac{1}{N}+\eps_1.
\end{array}
\right. 
$$
Since $1/N + \eps_1$ is not isolated from the left in $\mathrm{sp}(\tilde{\tilde{g}})$ and $A$ is simple, we have $\rho_1>0$.

By \eqref{u-1-close-D}, there is a positive contraction
$$ [u^*_1\chi_{1, \eps_1}(\tilde{\tilde{g}}) u_1] \in D $$
such that
\begin{equation}\label{down-to-D-1}
 \norm{ [u^*_1\chi_{1, \eps_1}(\tilde{\tilde{g}}) u_1] - u^*_1\chi_{1, \eps_1}(\tilde{\tilde{g}}) u_1 }_{2, \tr(A)} < 3\delta''<\delta_1. 
\end{equation}
With $\delta''$ sufficiently small, one has
$$\norm{ \eta_{\delta_1}([u^*_1\chi_{1, \eps_1}(\tilde{\tilde{g}}) u_1]) - u^*_1\eta_{\delta_1}(\chi_{1, \eps_1}(\tilde{\tilde{g}})) u_1 }_{2, \tr(A)} < \min\{\rho_1/8, \delta_1 \}.$$
Define
\begin{equation}\label{defn-down-to-D-m}
[u^*_1\eta_{\delta_1}(\chi_{1, \eps_1}(\tilde{\tilde{g}})) u_1] := \eta_{\delta_1}([u^*_1\chi_{1, \eps_1}(\tilde{\tilde{g}}) u_1]) \in D.
\end{equation}
Then
\begin{equation}\label{down-to-D-2}
 \norm{ [u^*_1\eta_{\delta_1}(\chi_{1, \eps_1}(\tilde{\tilde{g}})) u_1] - u^*_1\eta_{\delta_1}(\chi_{1, \eps_1}(\tilde{\tilde{g}})) u_1 }_{2, \tr(A)} < \min\{ \rho_1/8, \delta_1\}.
 \end{equation}
 Moreover (by \eqref{defn-down-to-D-m}), 
 \begin{equation}\label{down-to-D-3}
 \eta_{2\delta_1}([u^*_1\chi_{1, \eps_1}(\tilde{\tilde{g}}) u_1]) [u^*_1\eta_{\delta_1}(\chi_{1, \eps_1}(\tilde{\tilde{g}})) u_1]  = [u^*_1\eta_{\delta_1}(\chi_{1, \eps_1}(\tilde{\tilde{g}})) u_1]. 
 \end{equation}
 
 One should assume $\delta''$ is sufficiently small  that
 \begin{equation}\label{down-to-D-4} 
 \norm{ \eta_{\eps_1/4}([u_1^*\chi_{1, \eps_1}(\tilde{\tilde{g}})u_1]) - \eta_{\eps_1/4}(u_1^*\chi_{1, \eps_1}(\tilde{\tilde{g}})u_1) }_{2, \tr(A)} < \delta_1,
 \end{equation}
and
 \begin{equation}\label{down-to-D-5} 
 \norm{ \eta_{2\delta_1}([u_1^*\chi_{1, \eps_1}(\tilde{\tilde{g}})u_1]) - \eta_{2\delta_1}(u_1^*\chi_{1, \eps_1}(\tilde{\tilde{g}})u_1) }_{2, \tr(A)} < \delta_1. 
 \end{equation}
 
Note that (use \eqref{u-1-cond-3} in the third and fifth steps), 
\begin{eqnarray*}
 (u_1^* \chi_{1, \eps_1}(\tilde{\tilde{g}}) u_1)\chi_2(f)
 & \approx_{3N\eps_1} & (u_1^* \chi_{1, \eps_1}(\tilde{\tilde{g}}) u_1)\chi_{2, 3\eps_1}(f) \\
 &=&  (u_1^* \chi_{1, \eps_1}(\tilde{\tilde{g}}) u_1) \xi_{1, \eps_1}^-(f)\chi_{2, 3\eps_1}(f) \\
 & \approx_{\delta''}^{\norm{\cdot}_2} & (u_1^* \chi_{1, \eps_1}(\tilde{\tilde{g}}) u_1) \eta_{\delta''}(u_1^*\chi_{1, \eps_1}(\tilde{\tilde{g}})u_1) \xi_{1, \eps_1}^-(f)\chi_{2, 3\eps_1}(f) \\
& \approx_{\delta''} & \eta_{\delta''}(u_1^*\chi_{1, \eps_1}(\tilde{\tilde{g}})u_1) \xi_{1, \eps_1}^-(f)\chi_{2, 3\eps_1}(f)\\
 & \approx_{\delta''}^{\norm{\cdot}_2} & \xi_{1, \eps_1}^-(f)\chi_{2, 3\eps_1}(f) \\
 & = & \chi_{2, 3\eps_1}(f) \approx_{3N\eps_1} \chi_2(f).
 \end{eqnarray*}
With $\delta''$ sufficiently small, one has 
\begin{equation}\label{id-down-1}
[u^*_1\chi_{1, \eps_1}(\tilde{\tilde{g}}) u_1]\chi_2(f) \approx^{\norm{\cdot}_2}_{7N\eps_1} \chi_2(f). 
\end{equation}


Note that, by \eqref{down-to-D-2}, 
\begin{equation}\label{left-nbhd-1}
 \abs{\tau( [u_1^* \eta_{\delta_1}( \chi_{1, \eps_1}(\tilde{\tilde{g}})) u_1] ) - \tau( u_1^* \eta_{\delta_1}( \chi_{1, \eps_1}(\tilde{\tilde{g}})) u_1 )} < \rho_1/8,\quad \tau \in\tr(A).
 \end{equation}

With $\delta''$ sufficiently small, one has
\begin{equation}\label{left-nbhd-2}
 \tau(\eta_{\delta_1}(\chi_{1, \eps_1}(\tilde{\tilde{g}}))) \approx_{\rho_1/8} \tau(u^*_1\eta_{\delta_1}(\chi_{1, \eps_1}(\tilde{\tilde{g}}))u_1),\quad \tau \in \tr(A), 
 \end{equation}
and
\begin{equation}\label{swap-u-1}
\norm{ \eta_{\delta_1}(u_1^* \chi_{1, \eps_1}(\tilde{\tilde{g}})u_1) - u_1^* \eta_{\delta_1}(\chi_{1, \eps_1}(\tilde{\tilde{g}}))u_1}_2 < \delta_1/4.
\end{equation}

Hence, by \eqref{left-nbhd-1} and \eqref{left-nbhd-2},
\begin{eqnarray*}
\mathrm{d}_\tau(\chi_{2, \eps_1}(\tilde{\tilde{g}})) & \leq &  \tau(\eta_{\delta_1}(\chi_{1, \eps_1}(\tilde{\tilde{g}}))) -\rho_1/2 \\
& \approx_{\rho_1/8} & \tau(u^*_1\eta_{\delta_1}(\chi_{1, \eps_1}(\tilde{\tilde{g}}))u_1) - \rho_1/2 \\
& \approx_{\rho_1/8} & \tau([u^*_1\eta_{\delta_1}(\chi_{1, \eps_1}(\tilde{\tilde{g}}))u_1]) - \rho_1/2  \\
& < & \mathrm{d}_\tau([u^*_1\eta_{\delta_1}(\chi_{1, \eps_1}(\tilde{\tilde{g}}))u_1]) - \rho_1/2,
\end{eqnarray*}
for all $\tau \in \tr(A)$,
and therefore
\begin{equation}\label{lem-cond-1}
\mathrm{d}_\tau(\chi_{2, \eps_1}(\tilde{\tilde{g}})) < \mathrm{d}_\tau([u^*_1\eta_{\delta_1}(\chi_{1, \eps_1}(\tilde{\tilde{g}}))u_1]),\quad \tau \in\tr(A). 
\end{equation}

By \eqref{u-1-cond-3} (and \eqref{down-to-D-2}, \eqref{swap-u-1}), (note that $3\eps_1 < 1/N$ and $\delta''\ll \delta_1$)
\begin{eqnarray}\label{left-nbhd-3}
  [u^*_1\eta_{\delta_1}(\chi_{1, \eps_1}(\tilde{\tilde{g}}))u_1] (\xi^+_{2, \eps_1}(f)) & = & [u^*_1\eta_{\delta_1}(\chi_{1, \eps_1}(\tilde{\tilde{g}}))u_1]  (\xi^-_{1, \eps_1}(f))   (\xi^+_{2, \eps_1}(f))  \\
  & \approx_{\delta_1}^{\norm{\cdot}_2}& (u^*_1\eta_{\delta_1}(\chi_{1, \eps_1}(\tilde{\tilde{g}}))u_1)  (\xi^-_{1, \eps_1}(f))   (\xi^+_{2, \eps_1}(f)) \nonumber \\ 
  & \approx_{\delta_1}^{\norm{\cdot}_2}& \eta_{\delta_1}((u^*_1(\chi_{1, \eps_1}(\tilde{\tilde{g}}))u_1))  (\xi^-_{1, \eps_1}(f))   (\xi^+_{2, \eps_1}(f))  \nonumber \\ 
  & \approx_{\delta''}^{\norm{\cdot}_2} & (\xi^-_{1, \eps_1}(f))   (\xi^+_{2, \eps_1}(f)) \nonumber  \\
  & = &   \xi^+_{2, \eps_1}(f), \nonumber
  \end{eqnarray}
and therefore, one also has 
\begin{equation}\label{left-nbhd-4}
  [u^*_1\eta_{\delta_1}(\chi_{1, \eps_1}(\tilde{\tilde{g}}))u_1] (\xi^-_{2, \eps_1}(f)) \approx^{\norm{\cdot}_2}_{3\delta_1} \xi^-_{2, \eps_1}(f). 
\end{equation}  

Now, fixing $\delta''$, we have the contraction $u_1 \in A$.

Let us inductively assume that contractions $u_1, ..., u_k$, where $k \leq N-1$, have been constructed such that (note that \eqref{k-approx-bak} and \eqref{k-match-tr} are void if $k=1$)
\begin{equation}\label{k-approx-close-D}
\mathrm{dist}_{2, \tr(A)}(u_i^*\chi_{i, \eps_1}(\tilde{\tilde{g}}) u_i, (D)^+_1) < \min\{\eps_0, \delta_2, \eps/4\},\quad i=1, ..., k, \quad \quad \textrm{(\eqref{u-1-close-D} when $k=1$)}
\end{equation}
\begin{equation}\label{k-approx-bak}
 \norm{ \chi_{i-1}(f) (u_{i}^* \chi_{i, \eps_1}(\tilde{\tilde{g}}) u_{i}) - (u_{i}^* \chi_{i, \eps_1}(\tilde{\tilde{g}}) u_{i}) }_{2, \tr(A)} < 4 \delta_1 < \eps_1,\quad i=1, ..., k,
\end{equation}
\begin{equation}\label{k-approx-for}
\norm{ (u_{i}^* \chi_{i, \eps_1}(\tilde{\tilde{g}}) u_{i}) \chi_{i+1}(f) - \chi_{i+1}(f) }_{2, \tr(A)} < 7 N \eps_1 + 3^i\delta_1,\quad i=1, ..., k, \quad \quad \textrm{(\eqref{id-down-1} when $k=1$)}
\end{equation}
and
\begin{equation}\label{k-match-tr}
\norm{ (u_{i-1}^* \chi_{i, \eps_1}(\tilde{\tilde{g}}) u_{i-1})(u_{i}^* \chi_{i, \eps_1}(\tilde{\tilde{g}}) u_{i}) - (u_{i}^* \chi_{i, \eps_1}(\tilde{\tilde{g}}) u_{i}) }_{2, \tr(A)} < 6\delta_1,\quad i=1, ..., k, 
\end{equation}
\begin{equation}\label{k-match-tr-p}
\abs{\tau((u_i^* x u_i)^j) - \tau(x^j)} < \delta_1,\quad i =1, ..., k,\ j=1, ..., M,\  x\in (A)_1,\ \tau \in \tr(A). \quad \quad \textrm{(\eqref{u-1-match-tr-p} when $k=1$)} 
\end{equation}

Moreover, the contraction $u_k$ satisfies  
\begin{equation}\label{k-u-1-cond-2}
\eta_{\eps_1/4}(u^*_k \chi_{k, \eps_1}(\tilde{\tilde{g}}) u_k) \in^{\norm{\cdot}_2}_{\delta_1} \overline{\xi_{k, \eps_1}^+(f) A \xi_{k, \eps_1}^+(f)}, \quad \quad \textrm{(\eqref{u-1-cond-2} when $k=1$; note that $\delta''<\delta$)}
\end{equation}
and there are positive contractions $[u^*_k\chi_{k, \eps_1}(\tilde{\tilde{g}}) u_k], [u^*_k\eta_{\delta_1}(\chi_{k, \eps_1}(\tilde{\tilde{g}})) u_k] \in D$ such that
\begin{equation}\label{k-down-to-D-3}
 \eta_{2\delta_1}([u^*_k\chi_{k, \eps_1}(\tilde{\tilde{g}}) u_k]) [u^*_k\eta_{\delta_1}(\chi_{k, \eps_1}(\tilde{\tilde{g}})) u_k]  = [u^*_k\eta_{\delta_1}(\chi_{k, \eps_1}(\tilde{\tilde{g}})) u_k], \quad \quad \textrm{(\eqref{down-to-D-3} when $k=1$)}
 \end{equation}
 \begin{equation}\label{k-down-to-D-4} 
 \norm{ \eta_{\eps_1/4}([u_k^*\chi_{k, \eps_1}(\tilde{\tilde{g}})u_k]) - \eta_{\eps_1/4}(u_k^*\chi_{k, \eps_1}(\tilde{\tilde{g}})u_k) }_{2, \tr(A)} < \delta_1, \quad \quad \textrm{(\eqref{down-to-D-4} when $k=1$)}
 \end{equation}
 \begin{equation}\label{k-down-to-D-5} 
 \norm{ \eta_{2\delta_1}([u_k^*\chi_{k, \eps_1}(\tilde{\tilde{g}})u_k]) - \eta_{2\delta_1}(u_k^*\chi_{k, \eps_1}(\tilde{\tilde{g}})u_k) }_{2, \tr(A)} < \delta_1, \quad \quad \textrm{(\eqref{down-to-D-5} when $k=1$)}
 \end{equation}
\begin{equation}\label{k-lem-cond-1}
\mathrm{d}_\tau(\chi_{k+1, \eps_1}(\tilde{\tilde{g}})) < \mathrm{d}_\tau([u^*_k\eta_{\delta_1}(\chi_{k, \eps_1}(\tilde{\tilde{g}}))u_k]), \quad \tau \in\tr(A), \quad \quad \textrm{(\eqref{lem-cond-1} when $k=1$)}
\end{equation}
\begin{equation}\label{k-left-nbhd-3}
  [u^*_k\eta_{\delta_1}(\chi_{k, \eps_1}(\tilde{\tilde{g}}))u_k] (\xi^+_{k+1, \eps_1}(f)) \approx^{\norm{\cdot}_2}_{3^k\delta_1} \xi^+_{k+1, \eps_1}(f), \quad \quad \textrm{(\eqref{left-nbhd-3} when $k=1$)} 
\end{equation} 
and
\begin{equation}\label{k-left-nbhd-4}
  [u^*_k\eta_{\delta_1}(\chi_{k, \eps_1}(\tilde{\tilde{g}}))u_k] (\xi^-_{k+1, \eps_1}(f)) \approx^{\norm{\cdot}_2}_{3^k\delta_1} \xi^-_{k+1, \eps_1}(f). \quad \quad \textrm{(\eqref{left-nbhd-4} when $k=1$)}
\end{equation} 

Let us construct $u_{k+1}$.
Define
$$\left< \xi^-_{k+1, \eps_1}(f) \right> := [u^*_k\eta_{\delta_1}(\chi_{k, \eps_1}(\tilde{\tilde{g}}))u_k] (\xi^-_{k+1, \eps_1}(f)) \in \mathrm{Her}([u_k^* \eta_{\delta_1}(\chi_{k, \eps_1}(\tilde{\tilde{g}}))u_k] ) \cap D.$$
It follows from \eqref{k-left-nbhd-4} that
\begin{equation}\label{k-Xi-cut}
 \norm{\left< \xi^-_{k+1, \eps_1}(f) \right> - \xi^-_{k+1, \eps_1}(f)}_{2, \tr(A)} < 3^k\delta_1.  
 \end{equation}
Then, for all $\tau \in \tr(A)$, 
$$ \tau(\theta_{k+1, \eps_1}(f)) > d_\tau(\xi^-_{k+1, \eps_1}(f)) + 2 \gamma \geq d_\tau([u^*_k\eta_{\delta_1}(\chi_{k, \eps_1}(\tilde{\tilde{g}}))u_k]\xi^-_{k+1, \eps_1}(f)) + 2 \gamma  = \mathrm{d}_\tau(\left< \xi^-_{k+1, \eps_1}(f) \right>) + 2\gamma,$$
and hence 
$$\tau(\theta_{k+1, \eps_1}(\tilde{\tilde{g}})) \approx_{\gamma} \tau(\theta_{k+1, \eps_1}(f)) > \mathrm{d}_\tau( \left< \xi_{k+1, \eps_1}^-(f) \right>) + 2 \gamma. $$ 
In particular,
\begin{equation}\label{k-lem-cond-2}
\tau(\theta_{k+1, \eps_1}(\tilde{\tilde{g}})) >  \mathrm{d}_\tau( \left< \xi_{k+1, \eps_1}^-(f) \right> ),\quad \tau \in \tr(A). 
\end{equation}

Also define 
\begin{equation}\label{k+1-Xi-+}
\left< \xi^+_{k+1, \eps_1}(f) \right> := [u_k^*(\eta_{\delta_1}(\chi_{k, \eps_1}(\tilde{\tilde{g}}))u_k] (\xi^+_{k+1, \eps_1}(f)) \in \mathrm{Her}([u_k^*(\eta_{\delta_1}((\chi_{k, \eps_1}(\tilde{\tilde{g}})))u_k] ) \cap D.
\end{equation}
Then,
\begin{eqnarray*}
&& \tau(\eta_{\eps_1/2}(\chi_{k+1, \eps_1}( f )) ) \\
& < & \tau(\xi_{k+1, \eps_1}^+(f)) - 3 \gamma \quad \quad (\eqref{small-gap})\\
& \approx_{3^k\delta_1} & \tau( [u_k^*(\eta_{\delta_1}(\chi_{k, \eps_1}(\tilde{\tilde{g}}))u_k] (\xi^+_{k+1, \eps_1}(f)) ) - 3 \gamma \quad\quad (\eqref{k-left-nbhd-3}) \\
&= & \tau(  \left< \xi^+_{k+1, \eps_1}(f) \right> ) - 3 \gamma \\
& < & \mathrm{d}_\tau( \left< \xi^+_{k+1, \eps_1}(f) \right> ) - 3 \gamma, 
\end{eqnarray*}
and therefore
$$ \tau(\eta_{\eps_1/2}(\chi_{k+1, \eps_1}(\tilde{\tilde{g}}))) \approx_{2\gamma} \tau(\eta_{\eps_1/2}(\chi_{k+1, \eps_1}( f )) )  < d_\tau( [\xi^+_{k+1, \eps_1}(f)]) - 3 \gamma+ 3^k\delta_1.$$
In particular  (note that $3^N\delta_1 < \gamma /2$), 
\begin{equation}\label{k-lem-cond-3}
\tau(\eta_{\eps_1/2}(\chi_{k+1, \eps_1}(\tilde{\tilde{g}}))) < d_\tau( [\xi^+_{k+1, \eps_1}(f)]),\quad \tau \in\tr(A).
\end{equation}

With \eqref{k-lem-cond-1}, \eqref{k-lem-cond-2}, and \eqref{k-lem-cond-3}, by Lemma \ref{D-sandwich}, for any $\delta''>0$ (to be fixed later), there is a contraction $u_{k+1} \in A$ such that
\begin{equation}\label{k-u-2-cond-1}
 u_{k+1}^*\chi_{k+1, \eps_1}(\tilde{\tilde{g}})u_{k+1} \in^{\norm{\cdot}_{2}}_{\delta''}   \overline{[u_k^*\eta_{\delta'}(\chi_{k, \eps_1}(\tilde{\tilde{g}}))u_k] A [u_k^*\eta_{\delta'}(\chi_{k, \eps_1}(\tilde{\tilde{g}}))u_k] },
 \end{equation}
\begin{equation}\label{k-u-2-cond-2}
\eta_{\eps_1/4}(u^*_{k+1} \chi_{k+1, \eps_1}(\tilde{\tilde{g}}) u_{k+1}) \in^{\norm{\cdot}_2}_{\delta''} \overline{\left< \xi_{k+1, \eps_1}^+(f) \right> A \left< \xi_{k+1, \eps_1}^+(f) \right>},
\end{equation}
\begin{equation}\label{k-u-2-cond-3}
 \eta_{\delta''}(u_{k+1}^*\chi_{k+1, \eps_1}(\tilde{\tilde{g}})u_{k+1}) \left< \xi_{k+1, \eps_1}^-(f) \right>  \approx_{\delta''}^{\norm{\cdot}_2} \left< \xi_{k+1, \eps_1}^-(f) \right>,
 \end{equation}
\begin{equation}\label{u-2-cond-4}
\mathrm{dist}_{2, \tr(A)}(u_{k+1}du_{k+1}^*, D_1) < \delta'',\quad \mathrm{dist}_{2, \tr(A)}(u_{k+1}^*du_{k+1}, D_1) < \delta'',\quad d \in D_1,
\end{equation}
and
\begin{equation}\label{u-2-cond-5}
\norm{u_{k+1}u_{k+1}^* - 1}_{2, \tr(A)},\  \norm{u_{k+1}^*u_{k+1} - 1}_{2, \tr(A)} < \delta''.
\end{equation}
 
By \eqref{u-2-cond-4} and \eqref{u-2-cond-5}, with $\delta''$ sufficiently small,
\begin{equation}\label{k+1-approx-close-D}
\mathrm{dist}_{2, \tr(A)}(u_{k+1}^*\chi_{k+1, \eps_1}(\tilde{\tilde{g}}) u_{k+1}, (D)^+_1) < 3\delta'' < \min\{\eps_0, \delta_2, \eps/2\}.
\end{equation}
This verifies Assumption \eqref{k-approx-close-D} for $k+1$.

 Also note 
 \begin{eqnarray*}
 & & u_{k+1}^*( \chi_{k+1, \eps_1}(\tilde{\tilde{g}})) u_{k+1} \\
 & \approx_{\delta''}^{\norm{\cdot}_2} & (u_{k+1}^*( \chi_{k+1, \eps_1}(\tilde{\tilde{g}})) u_{k+1}) ( \eta_{2\delta_1}([u_k^* \chi_{k, \eps_1}(\tilde{\tilde{g}}) u_k] )) \quad \quad (\eqref{k-u-2-cond-1}, \eqref{k-down-to-D-3}) \\
 & \approx_{\delta_1}^{\norm{\cdot}_2} & (u_{k+1}^*( \chi_{k+1, \eps_1}(\tilde{\tilde{g}})) u_{k+1}) ( \eta_{2\delta_1}(u_k^* \chi_{k, \eps_1}(\tilde{\tilde{g}}) u_k )) \quad \quad (\eqref{k-down-to-D-5}) \\
 & \approx_{2\delta_1}^{\norm{\cdot}_2} & (u_{k+1}^*( \chi_{k+1, \eps_1}(\tilde{\tilde{g}})) u_{k+1}) ( \eta_{2\delta_1}(u_k^* \chi_{k, \eps_1}(\tilde{\tilde{g}}) u_k )) (u_k\chi_{k, \eps_1}(\tilde{\tilde{g}}) u_k) \\
 &\approx_{\delta_1+\delta''}& (u_{k+1}^*( \chi_{k+1, \eps_1}(\tilde{\tilde{g}})) u_{k+1}) (u^*_k\chi_{k, \eps_1}(\tilde{\tilde{g}}) u_k). \quad \quad (\eqref{k-u-2-cond-1}, \eqref{k-down-to-D-3}, \eqref{k-down-to-D-5})
\end{eqnarray*}
In other words,
\begin{equation*}
 (u_k^* \chi_{k, \eps_1}(\tilde{\tilde{g}}) u_k)(u_{k+1}^*(\chi_{k+1, \eps_1}(\tilde{\tilde{g}}) )u_{k+1}) \approx_{6\delta_1}^{\norm{\cdot}_2} u_{k+1}^*(\chi_{k+1, \eps_1}(\tilde{\tilde{g}}) )u_{k+1}.
 \end{equation*}
 This verifies Assumption \eqref{k-match-tr} for $k+1$.
 
 By \eqref{k-u-1-cond-2} and \eqref{k-u-2-cond-1}, \eqref{k-down-to-D-3}, \eqref{k-down-to-D-4},
\begin{eqnarray*}
\chi_k(f) (u_{k+1}^*(\chi_{k+1, \eps_1}(\tilde{\tilde{g}}) )u_{k+1}) & \approx_{\delta''} & \chi_k(f) \eta_{2\delta_1}([u_k^*\chi_{k, \eps_1}(\tilde{\tilde{g}})u_k]) (u_{k+1}^*(\chi_{k+1, \eps_1}(\tilde{\tilde{g}}) )u_{k+1}) \quad \quad (\eqref{k-down-to-D-3}) \\
& = & \chi_k(f) \eta_{\eps_1/4}([u_k^*\chi_{k, \eps_1}(\tilde{\tilde{g}})u_k]) (u_{k+1}^*(\chi_{k+1, \eps_1}(\tilde{\tilde{g}}) )u_{k+1})\\
& \approx^{\norm{\cdot}_2}_{\delta_1} & (\chi_k(f) \eta_{\eps_1/4}(u_k^*\chi_{k, \eps_1}(\tilde{\tilde{g}})u_k)) (u_{k+1}^*(\chi_{k+1, \eps_1}(\tilde{\tilde{g}}) )u_{k+1}) \quad \quad (\eqref{k-down-to-D-4}) \\
& \approx^{\norm{\cdot}_2}_{\delta_1} & \eta_{\eps_1/4}(u_k^*\chi_{k, \eps_1}(\tilde{\tilde{g}})u_k) (u_{k+1}^*(\chi_{k+1, \eps_1}(\tilde{\tilde{g}}) )u_{k+1}) \quad \quad (\eqref{k-u-1-cond-2}) \\
& \approx^{\norm{\cdot}_2}_{\delta_1} & \eta_{\eps_1/4}([u_k^*\chi_{k, \eps_1}(\tilde{\tilde{g}})u_k]) (u_{k+1}^*(\chi_{k+1, \eps_1}(\tilde{\tilde{g}}) )u_{k+1}) \quad \quad (\eqref{k-down-to-D-4}) \\
& = & u_{k+1}^*(\chi_{k+1, \eps_1}(\tilde{\tilde{g}}) )u_{k+1}.
\end{eqnarray*}
So,
$$ \chi_k(f) (u_{k+1}^*(\chi_{k+1, \eps_1}(\tilde{\tilde{g}}) )u_{k+1}) \approx^{\norm{\cdot}_2}_{4\delta_1} u_{k+1}^*(\chi_{k+1, \eps_1}(\tilde{\tilde{g}}) )u_{k+1}.  $$
This verifies Assumption \eqref{k-approx-bak} for $k+1$.

If $k+1 \leq N-1$, with the same argument as for \eqref{id-down-1}, 
\begin{eqnarray*}
 && (u_{k+1}^* \chi_{{k+1}, \eps_1}(\tilde{\tilde{g}}) u_{k+1})\chi_{k+2}(f) \\
 & \approx_{3N\eps_1} & (u_{k+1}^* \chi_{k+1, \eps_1}(\tilde{\tilde{g}}) u_{k+1})\chi_{k+2, 3\eps_1}(f) \\
 &=&  (u_{k+1}^* \chi_{k+1, \eps_1}(\tilde{\tilde{g}}) u_{k+1}) \xi_{k+1, \eps_1}^-(f)\chi_{k+2, 3\eps_1}(f) \\
 &\approx^{\norm{\cdot}_2}_{3^k\delta_1}&  (u_{k+1}^* \chi_{k+1, \eps_1}(\tilde{\tilde{g}}) u_{k+1}) \left< \xi_{k+1, \eps_1}^-(f) \right> \chi_{k+2, 3\eps_1}(f) \quad \quad (\eqref{k-Xi-cut}) \\
 & \approx_{\delta''}^{\norm{\cdot}_2} & (u_{k+1}^* \chi_{k+1, \eps_1}(\tilde{\tilde{g}}) u_{k+1}) \eta_{\delta''}(u_{k+1}^*\chi_{k+1, \eps_1}(\tilde{\tilde{g}})u_{k+1}) \left< \xi_{k+1, \eps_1}^-(f) \right> \chi_{k+2, 3\eps_1}(f) \quad \quad (\eqref{k-u-2-cond-3}) \\
& \approx_{\delta''} & \eta_{\delta''}(u_{k+1}^*\chi_{k+1, \eps_1}(\tilde{\tilde{g}})u_{k+1}) \left< \xi_{k+1, \eps_1}^-(f)\right> \chi_{k+2, 3\eps_1}(f)\\
 & \approx_{\delta''}^{\norm{\cdot}_2} &  \left<  \xi_{k+1, \eps_1}^-(f) \right> \chi_{k+2, 3\eps_1}(f) \quad \quad ( \eqref{k-u-2-cond-3}) \\
 & \approx^{\norm{\cdot}_2}_{3^k\delta_1} & \chi_{k+2, 3\eps_1}(f) \approx_{3N\eps_1} \chi_{k+2}(f).  \quad \quad (\eqref{k-Xi-cut})
 \end{eqnarray*}
Thus, 
\begin{equation*}
 (u_{k+1}^*(\chi_{k+1, \eps_1}(\tilde{\tilde{g}}) )u_{k+1}) \chi_{k+2}(f) \approx_{7N\eps_1 + 3^{k+1}\delta_1}^{\norm{\cdot}_2} \chi_{k+2}(f).
\end{equation*}
This verifies Assumption \eqref{k-approx-for} for $k+1$ (which is void if $k+1 = N$).

With $\delta''$ sufficiently small, one has  
\begin{equation*}
\abs{\tau((u_{k+1}^* x u_{k+1})^j) - \tau(x^j)} < \delta_1,\quad j=1, ..., M,\  x\in (A)_1,\ \tau \in \tr(A). 
\end{equation*}
This verifies \eqref{k-match-tr-p} for $k+1$.

If $k+1 \leq N-1$, let us verify that (with $\delta''$ sufficiently small), the contraction $u_{k+1}$ satisfies the inductive assumptions \eqref{k-u-1-cond-2}, \eqref{k-down-to-D-3}, \eqref{k-down-to-D-4}, \eqref{k-down-to-D-5}, \eqref{k-lem-cond-1}, \eqref{k-left-nbhd-3}, and \eqref{k-left-nbhd-4} for $k+1$.
 
By \eqref{k+1-Xi-+} and noting that $ [u_k^*(\eta_{\delta_1}(\chi_{k, \eps_1}(\tilde{\tilde{g}}))u_k]$ and $ (\xi^+_{k+1, \eps_1}(f))$ commute (both are in $D$), one has
$$ \overline{\left< \xi_{k+1, \eps_1}^+(f) \right> A \left< \xi_{k+1, \eps_1}^+(f) \right>} \subseteq \overline{ \xi_{k+1, \eps_1}^+(f) A  \xi_{k+1, \eps_1}^+(f) }.$$
 Thus, Assumption \eqref{k-u-1-cond-2} for $k+1$ follows from \eqref{k-u-2-cond-2}. 
 
For the other assumptions, let us repeat the argument of $u_1$:  Set 
$$ \rho_{k+1}:=\min\{\tau(\rho_{k+1, \delta_1}(\tilde{\tilde{g}})): \tau \in \tr(A)\},$$
where
$$\rho_{k+1, \delta_1} = \left\{
\begin{array}{ll}
0, & t\leq \frac{k+1}{N}+\eps_1-\delta_1, \\
\mathrm{linear}, & \frac{k+1}{N} + \eps_1-\delta_1 \leq t \leq \frac{k+1}{N} + \eps_1-\delta_1/2, \\
1, & t = \frac{k+1}{N} + \eps_1-\delta_1/2, \\
\mathrm{linear}, & \frac{k+1}{N} + \eps_1-\delta_1/2 \leq t \leq \frac{k+1}{N} + \eps_1, \\
0, & t \geq \frac{k+1}{N}+\eps_1.
\end{array}
\right. 
$$
Since $(k+1)/N + \eps_1$ is not isolated from the left in $\mathrm{sp}(\tilde{\tilde{g}})$ and $A$ is simple, we have that $\rho_1>0$.

By \eqref{k+1-approx-close-D}, with a sufficiently small $\delta''$, there is a positive contraction 
$$ [u^*_{k+1}\chi_{k+1, \eps_1}(\tilde{\tilde{g}}) u_{k+1}] \in D $$
such that
\begin{equation*}
 \norm{ [u^*_{k+1}\chi_{k+1, \eps_1}(\tilde{\tilde{g}}) u_{k+1}] - u^*_{k+1}\chi_{k+1, \eps_1}(\tilde{\tilde{g}}) u_{k+1} }_{2, \tr(A)} < 3\delta''<\delta_1. 
\end{equation*}
With $\delta''$ sufficiently small, one has
$$\norm{ \eta_{\delta_1}([u^*_{k+1}\chi_{k+1, \eps_1}(\tilde{\tilde{g}}) u_{k+1}]) - u^*_{k+1}\eta_{\delta_1}(\chi_{k+1, \eps_1}(\tilde{\tilde{g}})) u_{k+1} }_{2, \tr(A)} < \min\{\rho_{k+1}/8, \delta_1 \}.$$
Define
\begin{equation}\label{k+1-defn-down-to-D-m}
[u^*_{k+1}\eta_{\delta_1}(\chi_{{k+1}, \eps_1}(\tilde{\tilde{g}})) u_{k+1}] := \eta_{\delta_1}([u^*_{k+1}\chi_{k+1, \eps_1}(\tilde{\tilde{g}}) u_{k+1}]) \in D.
\end{equation}
Then
\begin{equation}\label{k+1-down-to-D-2}
 \norm{ [u^*_{k+1}\eta_{\delta_1}(\chi_{k+1, \eps_1}(\tilde{\tilde{g}})) u_{k+1}] - u^*_{k+1}\eta_{\delta_1}(\chi_{k+1, \eps_1}(\tilde{\tilde{g}})) u_{k+1} }_{2, \tr(A)} < \min\{ \rho_{k+1}/8, \delta_1\}.
 \end{equation}
 Moreover (by \eqref{k+1-defn-down-to-D-m}), 
 \begin{equation}\label{k+1-down-to-D-3}
 \eta_{2\delta_1}([u^*_{k+1}\chi_{k+1, \eps_1}(\tilde{\tilde{g}}) u_{k+1}]) [u^*_{k+1}\eta_{\delta_1}(\chi_{k+1, \eps_1}(\tilde{\tilde{g}})) u_{k+1}]  = [u^*_{k+1}\eta_{\delta_1}(\chi_{k+1, \eps_1}(\tilde{\tilde{g}})) u_{k+1}]. 
 \end{equation}
 This verifies 
 \eqref{k-down-to-D-3} for $k+1$.
 
 One should assume $\delta''$ is sufficiently small  that
 \begin{equation}\label{k+1-down-to-D-4} 
 \norm{ \eta_{\eps_1/4}([u_{k+1}^*\chi_{k+1, \eps_1}(\tilde{\tilde{g}})u_{k+1}]) - \eta_{\eps_1/4}(u_{k+1}^*\chi_{k+1, \eps_1}(\tilde{\tilde{g}})u_{k+1}) }_{2, \tr(A)} < \delta_1,
 \end{equation}
and
 \begin{equation}\label{k+1-down-to-D-5} 
 \norm{ \eta_{2\delta_1}([u_{k+1}^*\chi_{k+1, \eps_1}(\tilde{\tilde{g}})u_{k+1}]) - \eta_{2\delta_1}(u_{k+1}^*\chi_{k+1, \eps_1}(\tilde{\tilde{g}})u_{k+1}) }_{2, \tr(A)} < \delta_1. 
 \end{equation}
 This verifies \eqref{k-down-to-D-4} and \eqref{k-down-to-D-5} for $k+1$.
 
 Note that, by \eqref{k+1-down-to-D-2}, 
\begin{equation}\label{k+1-left-nbhd-1}
 \abs{\tau( [u_{k+1}^* \eta_{\delta_1}( \chi_{k+1, \eps_1}(\tilde{\tilde{g}})) u_{k+1}] ) - \tau( u_{k+1}^* \eta_{\delta_1}( \chi_{k+1, \eps_1}(\tilde{\tilde{g}})) u_{k+1} )} < \rho_{k+1}/8,\quad \tau \in\tr(A).
 \end{equation}
 
 With $\delta''$ sufficiently small, one has
\begin{equation}\label{k+1-left-nbhd-2}
 \tau(\eta_{\delta_1}(\chi_{k+1, \eps_1}(\tilde{\tilde{g}}))) \approx_{\rho_{k+1}/8} \tau(u^*_{k+1}\eta_{\delta_1}(\chi_{k+1, \eps_1}(\tilde{\tilde{g}}))u_{k+1}),\quad \tau \in \tr(A), 
 \end{equation}
and
\begin{equation}\label{k+1-swap-u-1}
\norm{ \eta_{\delta_1}(u_{k+1}^* \chi_{k+1, \eps_1}(\tilde{\tilde{g}})u_{k+1}) - u_{k+1}^* \eta_{\delta_1}(\chi_{k+1, \eps_1}(\tilde{\tilde{g}}))u_{k+1}}_{2, \tr(A)} < \delta_1/4.
\end{equation}
Hence, by \eqref{k+1-left-nbhd-1} and \eqref{k+1-left-nbhd-2},
\begin{eqnarray*}
\mathrm{d}_\tau(\chi_{k+2, \eps_1}(\tilde{\tilde{g}})) & \leq &  \tau(\eta_{\delta_1}(\chi_{k+1, \eps_1}(\tilde{\tilde{g}}))) -\rho_{k+1}/2 \\
& \approx_{\rho_{k+1}/8} & \tau(u^*_{k+1}\eta_{\delta_1}(\chi_{k+1, \eps_1}(\tilde{\tilde{g}}))u_{k+1}) - \rho_{k+1}/2 \\
& \approx_{\rho_{k+1}/8} & \tau([u^*_{k+1}\eta_{\delta_1}(\chi_{k+1, \eps_1}(\tilde{\tilde{g}}))u_{k+1}]) - \rho_{k+1}/2  \\
& < & \mathrm{d}_\tau([u^*_{k+1}\eta_{\delta_1}(\chi_{k+1, \eps_1}(\tilde{\tilde{g}}))u_{k+1}]) - \rho_{k+1}/2,
\end{eqnarray*}
for all $\tau \in \tr(A)$,
and therefore
\begin{equation}\label{k+1-lem-cond-1}
\mathrm{d}_\tau(\chi_{k+2, \eps_1}(\tilde{\tilde{g}})) < \mathrm{d}_\tau([u^*_{k+1}\eta_{\delta_1}(\chi_{k+1, \eps_1}(\tilde{\tilde{g}}))u_{k+1}]). 
\end{equation}
This verifies \eqref{k-lem-cond-1} for $k+1$.

By \eqref{k-u-2-cond-3} and \eqref{k-Xi-cut} (and \eqref{k+1-down-to-D-2}, \eqref{k+1-swap-u-1}), (note that $3\eps_1 < 1/N$ and $\delta''\ll \delta_1$)
\begin{eqnarray}\label{k+1-left-nbhd-3}
  & & [u^*_{k+1}\eta_{\delta_1}(\chi_{k+1, \eps_1}(\tilde{\tilde{g}}))u_{k+1}] (\xi^+_{k+2, \eps_1}(f)) \\
  & = & [u^*_{k+1}\eta_{\delta_1}(\chi_{k+1, \eps_1}(\tilde{\tilde{g}}))u_{k+1}]  (\xi^-_{k+1, \eps_1}(f))   (\xi^+_{k+2, \eps_1}(f))  \nonumber \\
  & \approx_{\delta_1}^{\norm{\cdot}_2}& (u^*_{k+1}\eta_{\delta_1}(\chi_{k+1, \eps_1}(\tilde{\tilde{g}}))u_{k+1})  (\xi^-_{k+1, \eps_1}(f))   (\xi^+_{k+2, \eps_1}(f))  \quad\quad ( \eqref{k+1-down-to-D-2}) \nonumber \\ 
  & \approx_{\delta_1}^{\norm{\cdot}_2}& \eta_{\delta_1}((u^*_{k+1}(\chi_{k+1, \eps_1}(\tilde{\tilde{g}}))u_{k+1}))  (\xi^-_{k+1, \eps_1}(f))   (\xi^+_{k+2, \eps_1}(f)) \quad \quad (\eqref{k+1-swap-u-1}) \nonumber \\ 
  & \approx_{3^k \delta_1}^{\norm{\cdot}_2}& \eta_{\delta_1}((u^*_{k+1}(\chi_{k+1, \eps_1}(\tilde{\tilde{g}}))u_{k+1}))  \left< \xi^-_{k+1, \eps_1}(f) \right>   (\xi^+_{k+2, \eps_1}(f)) \quad \quad (\eqref{k-Xi-cut}) \nonumber \\
  & \approx_{\delta''}^{\norm{\cdot}_2} & \left< \xi^-_{k+1, \eps_1}(f) \right>   (\xi^+_{k+2, \eps_1}(f)) \quad \quad (\eqref{k-u-2-cond-3}) \nonumber  \\
  & \approx_{3^k \delta_1}^{\norm{\cdot}_2} & ( \xi^-_{k+1, \eps_1}(f) )   (\xi^+_{k+2, \eps_1}(f)) \quad \quad (\eqref{k-Xi-cut}) \nonumber  \\
  & = &   \xi^+_{k+2, \eps_1}(f), \nonumber
  \end{eqnarray}
and therefore, one also has 
\begin{equation}\label{k+1-left-nbhd-4}
  [u^*_{k+1}\eta_{\delta_1}(\chi_{k+1, \eps_1}(\tilde{\tilde{g}}))u_{k+1}] (\xi^-_{k+2, \eps_1}(f)) \approx^{\norm{\cdot}_2}_{3^{k+1}\delta_1} \xi^-_{k+2, \eps_1}(f). 
\end{equation}  
This verifies \eqref{k-left-nbhd-3} and \eqref{k-left-nbhd-4} for $k+1$. Fix $\delta''$, and we obtain the desired $u_{k+1}$.

By induction, there are contractions $u_1, u_2, ..., u_{N} \in A$ such that (note that $3^N\delta_1 < \eps_1$)
\begin{equation}\label{N-approx-close-D}
\mathrm{dist}_{2, \tr(A)}(u_i^*\chi_{i, \eps_1}(\tilde{\tilde{g}}) u_i, (D)^+_1) < \min\{\eps_0, \delta_2, \eps/4\} \leq \eps_0,\quad i=1, ..., N,
\end{equation}
\begin{equation}\label{N-approx-bak}
 \norm{ \chi_{i-1}(f) (u_{i}^* \chi_{i, \eps_1}(\tilde{\tilde{g}}) u_{i}) - (u_{i}^* \chi_{i, \eps_1}(\tilde{\tilde{g}}) u_{i}) }_{2, \tr(A)} < 4 \delta_1 < \eps_0,\quad i=2, ..., N,
\end{equation}
\begin{equation}\label{N-approx-for}
\norm{ (u_{i}^* \chi_{i, \eps_1}(\tilde{\tilde{g}}) u_{i}) \chi_{i+1}(f) - \chi_{i+1}(f) }_{2, \tr(A)} < 7 N \eps_1 + 3^i\delta_1 < 8N\eps_1 < \eps_0,\quad i=1, ..., N-1, 
\end{equation}
\begin{equation}\label{N-match-tr}
\norm{ (u_{i-1}^* \chi_{i, \eps_1}(\tilde{\tilde{g}}) u_{i-1})(u_{i}^* \chi_{i, \eps_1}(\tilde{\tilde{g}}) u_{i}) - (u_{i}^* \chi_{i, \eps_1}(\tilde{\tilde{g}}) u_{i}) }_{2, \tr(A)} < 6\delta_1 < \delta_0,\quad i=1, ..., N, 
\end{equation}
and
\begin{equation}\label{N-match-tr-p}
\abs{\tau((u_i^* x u_i)^j) - \tau(x^j)} < \delta_1 < \delta_0,\quad i =1, ..., N,\ j=1, ..., M,\  x\in (A)_1,\ \tau \in \tr(A).\end{equation}

Define
$$g:= \frac{1}{N}(u_1^*\chi_{1, \eps_1}(\tilde{\tilde{g}}) u_1 + \cdots + u_N^*\chi_{N, \eps_1}(\tilde{\tilde{g}}) u_N).$$
Then, by \eqref{N-approx-close-D}, \eqref{N-approx-bak} and \eqref{N-approx-for}, it follows from Lemma \ref{hereditary-sets-tr-app} that
\begin{equation}\label{almost-g-1}
 \norm{ f - g}_{2, \tr(A)} < \eps/2.
 \end{equation}

Note that
$$(\tilde{\tilde{g}} - \eps_1)_+ = \frac{1}{N}(\chi_{1, \eps_1}(\tilde{\tilde{g}}) + \cdots + \chi_{N, \eps_1}(\tilde{\tilde{g}}) ).$$
By \eqref{N-match-tr} and  \eqref{N-match-tr-p}, it follows from Lemma \ref{hereditary-sets-preserve-tr} that 
\begin{equation}\label{almost-g-2}
\abs{ \tau(\chi_{\frac{1}{2}+\eps_1, \delta}((\tilde{\tilde{g}} - \eps_1))) - \tau(\chi_{\frac{1}{2}+\eps_1, \delta}(g))} < \eps/4,\quad \tau\in \tr(A). 
\end{equation}

By \ref{N-approx-close-D} again, one has
$$\mathrm{dist}_{2, \tr(A)}(g, (D)_1^+) < \min\{\delta_2, \eps/4\};$$
together with \eqref{almost-g-1}, \eqref{almost-g-2}, and the choice of $\delta_2$ (\eqref{choice-delta-2}),
there is a positive contraction in $D$, still denoted by $g$, such that 
$$\norm{f - g }_{2, \tr(A)} < \eps/2 + \eps/4 = 3\eps/4$$ 
and,
$$ \abs{ \tau(\chi_{\frac{1}{2}+\eps_1, \delta}((\tilde{\tilde{g}} - \eps_1))) - \tau(\chi_{\frac{1}{2}+\eps_1, \delta}(g))} < \eps/4 + \eps/4 = \eps/2,\quad \tau\in \tr(A) $$
By \eqref{small-strech}, one has
$$ \tau(\chi_{\frac{1}{2}+\eps_1, \delta}(g)) <\eps/2 + \eps/2 = \eps.$$
Stretching $g$ to move $\frac{1}{2} + \eps_1$ to $\frac{1}{2}$ (and note that $\eps_1< \eps/2$), it satisfies the desired approximations \eqref{positive-s1} and \eqref{positive-s2}.
\end{proof}

As a consequence of Theorem \ref{prop-S}, one has the following corollary: 

\begin{cor}\label{Z-SBP}
Let $A$ be a simple AH algebra with diagonal maps, or let $A = \mathrm{C}(X)\rtimes\Gamma$, where $(X, \Gamma)$ is free, minimal, and has the (URP) and (COS). Then $(D, A)$ has Property (C) if, and only if, $(D, \tr(A)|_D)$ has the (SBP).
%
\end{cor}

\begin{proof}
Assume that $(D, A)$ has Property (C). By Theorem \ref{Property-D}, the C*-algebra pair  $(D, A)$ has Property (E). Note that $A$ has Property (S) (Theorem \ref{S-holds}).  By Theorem \ref{prop-S}, $(D, \tr(A))$ has the (SBP).

The other implication follows from Theorem \ref{SBP=>C}.
\end{proof}

\begin{rem}
Let $A$ be an AH algebra with diagonal map or a crossed product C*-algebra as above. It would be an interesting question whether $\mathcal Z$-absorption of $A$, which is a sole C*-algebra property, would imply Property (C) of the pair $(D, A)$.  
\end{rem}

\bibliographystyle{plainurl}
\bibliography{operator_algebras}

\end{document}